\documentclass[12p]{amsart}
\usepackage[top=1.5in, bottom=1.5in, left=1.5in, right=1.5in]{geometry}
\usepackage{fourier} 
\usepackage{latexsym}
\usepackage{amssymb}
\usepackage{amsmath}
\usepackage{amscd}
\usepackage{graphicx}
\usepackage{float}
\usepackage{mathrsfs} 
\usepackage{enumerate}
\usepackage{tikz}

\tikzset{node distance=1.5cm, auto}

\newtheorem{theorem}{Theorem}[section]
\newtheorem{proposition}[theorem]{Proposition}
\newtheorem{lemma}[theorem]{Lemma}
\newtheorem{corollary}[theorem]{Corollary}

\theoremstyle{definition}
\newtheorem{definition}[theorem]{Definition}
\newtheorem{algorithm}[theorem]{Algorithm}
\theoremstyle{remark}
\newtheorem{example}[theorem]{Example}
\newtheorem{remark}[theorem]{Remark}

\def\<{\langle}
\def\>{\rangle}

\def\RR{\mathbb{R}}

\def\CC{\mathbb{C}}
\def\ZZ{\mathbb{Z}}
\def\T{\mathbf{T}}
\def\H{\mathcal{H}}
\def\V{\mathcal{V}}

\def\C{\mathcal{C}}
\def\D{\mathcal{D}}
\def\N{\mathcal{N}}

\def\Arg{\operatorname{Arg}}
\def\pr{\operatorname{pr}}

\def\area{\operatorname{area}}

\def\N{\mathcal{N}}

\title{On dimer models and coamoebas}
\date{\today}
\author{Jens Forsg{\aa}rd}
\address{Department of Mathematics \\ Texas A\&M University \\
College Station, TX 77843.}
\email{jensf@math.tamu.edu}

\begin{document}

\begin{abstract}
We describe the relationship between dimer models on the real two-torus and coamoebas
of curves in $(\CC^\times)^2$. 
We show, inter alia, that the dimer model obtained from the shell of the coamoeba is a
deformation retract of the closed coaomeba if and only if the number of connected components 
of the complement of the closed coamoeba is maximal. Furthermore, we show that in general the closed coamoeba of the characteristic polynomial of a dimer model does not have the maximal number of
components of its complement.
\end{abstract}

\maketitle
\section{Introduction}

Dimer models, i.e., bipartite graphs embedded in an orientable surface which admit perfect matchings, appeared in the 1930s as
statistical models for the absorption of di-atomic molecules (dimers) on a crystal surface.
About a decade and a half ago it was observed that dimer models have
vast applications within mirror symmetry and string theory. In particular, dimer models
embedded in the real two-torus $\T^2$ are related to brane tilings of toric singular 
Calabi--Yau threefolds, see \cite{FHKV08, Gul08, HV07} and the references therein.

Let $G\subset \T^2$ be a dimer model. We can form the Kasteleyn matrix 
(or weighted adjacency matrix) associated with $G$, see, e.g., \cite{KO06}.
Its determinant, which is a bivariate polynomial, is known as the characteristic polynomial
of $G$, and its Newton polygon $\N$ is known as the characteristic polygon of $G$.
In the physics literature, the characteristic polygon is called the toric diagram, see \cite{HV07}.

The inverse problem, to construct a dimer model with a prescribed characteristic polygon $\N$,
has been addressed in a number of articles. 
The first proposed solution, now known as the Hanany--Vegh algorithm, assumed the existence of an 
oriented \emph{admissible} hyperplane arrangement on $\T^2$ \emph{dual} to the polygon $\N$
\cite{HV07, Sti08}, see \S\ref{sec:graphs} for definitions.
An alternative algorithm not subject to any additional assumptions
has been given by Gulotta \cite{Gul08}.

The present work has its origin in a series of papers by Futaki--Ueda and Ueda--Yamzaki
\cite{FU10, UY11, UY13}, who study three polygons $\N$ in detail:
the unit simplex, the unit square, and one special polygon with five vertices
(the case $k=1$ in Example~\ref{ex:HananyVeghExample}).
Their main observation was that \emph{in these three examples} the admissible hyperplane arrangement
can be taken as the shell $\H$ of the coamoeba $\overline{\C}$ of the characteristic 
polynomial of $G$. In addition, the dimer model $G$ can be realized as a deformation retract of the 
coamoeba $\overline{\C}$.
The main purpose of this work is to explain the relationship between the dimer model 
$G$ and the coamoeba $\overline{\C}$.
As is common in the amoeba literature, we take the approach of Gelfand, Kapranov, and Zelevinsky \cite{GKZ94}
and study a family of polynomials with fixed support whose Newton polygon is $\N$.

In general, the shell $\H$ of the coamoeba $\C$ is not an admissible hyperplane arrangement. 
Even worse, we provide a polygon which does not admit any dual 
admissible hyperplane arrangement, see Example \ref{ex:HananyVeghExample}.
This settles a question arising from \cite{Sti08}.
One point of this article is that the notion of admissibility is a red herring.
We consider instead the notion of \emph{index}, which refines the notion of admissibility.
Each dual hyperplane arrangement $\H$ of a polygon $\N$ has an associated index map. That is,
there is a map $\iota \colon \pi_0(\T^2\setminus \H) \rightarrow \ZZ$ subject to a certain crossing
rule, see \S \ref{sec:graphs}. It is not hard to show that the hyperplane arrangement $\H$
is admissible if and only if $|\iota(P)| \leq 1$ for all $P \in \pi_0(\T^2\setminus \H)$.
Also, if $\H$ is admissible, then the number of cells $P \in \pi_0(\T^2\setminus \H)$ 
of index zero is exactly $2 \area(\N)$. 
The main technical result of this paper is the following characterization.

\begin{theorem}
\label{thm:ShellVolP2}
A generic oriented dual hyperplane arrangement $\H$ of $\N$ has $2\area(\N)$-many cells of index zero
if and only if\/ $|\iota(P)| \leq 2$ and for each cell with $|\iota(P)| = 2$ it holds that
$P$ is a triangle.
\end{theorem}

We provide a generalization of the Hanany--Vegh algorithm, the \emph{index graph algorithm},
which does not require the dual hyperplane arrangement (i.e., the shell of the coamoeba) to be admissible, 
see \S\ref{sec:graphs}. This algorithm is equivalent to an algorithm set forward by Stienstra \cite{Sti08}, but is formulated in terms of the index map $\iota$. 
The justification for why we rewrite Stienstra's algorithm in this manner,
is that we need to be able to handle the case set forward in Theorem \ref{thm:ShellVolP2}.
The reader familiar with Yang--Baxter modifications (see \S\ref{ssec:YangBaxter})
will realize that the moral of Theorem~\ref{thm:ShellVolP2} is that rather
than requiring $\H$ to be an admissible arragnement, we should require that 
$\H$ has the correct number of cells of index zero. 
Our main result is the following theorem.

\begin{theorem}
\label{thm:DeformationRetract}
Let $f$ be a bivariate polynomial with Newton polygon $\N$. 
Then, the dimer model $G$ obtained from the shell $\H$ by the 
index graph algorithm and Yang--Baxter modifications 
is a deformation retract of the coamoeba $\overline{\C}$
if and only if 
the cardinality of\/ $\pi_0\big(\T^2\setminus \overline{\C}\big)$
is equal to $2\area(\N)$.
\end{theorem}

It was shown in \cite{FJ14} that  $2 \area(\N)$ 
is an upper bound on the cardinality of 
$\pi_0\big(\T^2\setminus \overline{\C}\big)$.
Hence, the dimer model $G$ is a deformation retract of $\overline{\C}$
if and only if the cardinality of $\pi_0\big(\T^2\setminus \overline{\C}\big)$
is maximal.

In the examples studied by Futaki, Ueda, and Yamazaki 
the polynomial $f$ appearing in
Theorem~\ref{thm:DeformationRetract} was taken as the 
characteristic polynomial of the dimer model $G$.
The characteristic polynomial defines a Harnack curve and, by
recent results of Lang \cite{Lan15}, the complement of the coamoeba of a Harnack curve has few 
connected components. 
In particular, on can not in general obtain the dimer model as a deformation retract of the coamoeba
of its characteristic polynomial; for an explicit example see Remark~\ref{rem:HarnackCurves}. 
The complement of the coamoeba of the characteristic polynomial may have the maximal 
number of connected components only if $\N$ is \emph{sparse along edges}, meaning that the vertices
of $\N$ are the only integer points on its boundary.

Theorem \ref{thm:DeformationRetract} raises the question of which polygons $\N$ admit a 
polynomial $f$ with Newton polygon $\N$ such that the complement of the coamoeba
$\overline{\C}$ has the maximal number of connected components.
As of this writing, the strongest result in this direction was obtained in \cite{FJ15}. 
It concerns the case when 
$f$ is supported on a (possibly degenerate) \emph{circuit}.
In the bivariate case, studying polynomials supported on a circuit is equivalent to studying tetranomials.
This is, in turn, equivalent to assuming that the Newton polygon $\N$ is either a triangle or
a quadrilateral.
From Theorem \ref{thm:DeformationRetract} 
and \cite{FJ15} we obtain the following result.

\begin{corollary}
\label{cor:HVforCircuits}
Let $f$ be a generic bivariate polynomial supported on a circuit.
Then, the cardinality of\/ $\pi_0(\T^2 \setminus \overline{\C})$ is maximal.
In particular, the dimer model $G$ obtained from $\H$ by the Hanany--Vegh algorithm
and Yang--Baxter modifications is a deformation retract of the coamoeba $\overline{\C}$.
\end{corollary}

Let us also emphasize Remark~\ref{rem:SingularPoints} where, in the circuit case, we find
that the argument map induces an explicit bijection between the critical points of the polynomial $f$ 
and the gauge groups in the quiver theory of the dimer model. 
That these two sets are of equal cardinality is known in the general case,
see \cite{FHKV08}. However, this is to the best of our knowledge the first explicit bijection
appearing in the literature.

\section{The coamoeba and the shell}
Let $A = \{\alpha_1, \dots, \alpha_N\} \subset \ZZ^2$ be a finite set of cardinality $N$. Consider a 
bivariate polynomial
\begin{equation}
\label{eqn:f}
f(z) = \sum_{k=1}^N x_k\,z^{\alpha_k}.
\end{equation}
We will identify $f$ with its coefficient vector $(x_1, \dots, x_N)$, and we will assume that the
representation of $f$ is minimal in the sense that $f \in (\CC^\times)^A$. Hence, $A$ is the \emph{support} 
of the polynomial $f$. The \emph{Newton polygon} of $f$, denoted $\N$, is the convex hull of $A$ when embedded in $\RR^2 = \RR\otimes \ZZ^2$. 

Let $\Gamma$ be a face of $\N$, which we denote by $\Gamma \prec \N$. The image of $f$ under the projection $\pr_\Gamma\colon (\CC^\times)^A \rightarrow (\CC^\times)^{\Gamma\, \cap\, A}$ is called the \emph{truncation} of $f$ to the face $\Gamma$, and is denoted $f_\Gamma$.
Let $v_1, \dots, v_m \in \ZZ^2$ denote the vertices of $\N$ cyclically ordered counterclockwise on the boundary of $\N$.
If $\Gamma$ is the facet with endpoints $v_k$ and $v_{k+1}$, where indices should be understood
modulo $m$, then we will identify $\Gamma$ with the vector $\Gamma = v_{k+1} - v_k$.
Let
\begin{equation}
\label{eqn:NormalVector}
\gamma = M\,\Gamma,
\end{equation}
where $M$ acts by clockwise rotation by the angle $\pi/2$.
That is, $\gamma$ is the outward pointing integer normal vector of $\Gamma$ 
whose integer length is equal to that of $\Gamma$. Note that $M$ restricts to a $\ZZ$-module
automorphism of $\ZZ^2$.

The coamoeba $\C$ of an algebraic variety $V\subset (\CC^\times)^n$ is defined as its image under the componentwise argument mapping $\Arg\colon (\CC^\times)^n \rightarrow \T^n$.
That is, $\C = \Arg(V)$.
Here, $\T = \RR/2\pi \ZZ$.
In this paper we are only concerned with the case when $V$ is a 
(not necessarily irreducible) curve in $(\CC^\times)^2$. 
In this case, if $\Gamma$ is a facet of $\N$,
then the truncation $f_\Gamma$ has a pseudo-homogeneity encoded by the normal
vector $\gamma$ of $\Gamma$;
the coamoeba of $f_\Gamma$, denoted $\C_\Gamma$, is a family of lines in $\T^2$ whose 
directional vector (when viewed in the universal covering $\RR^2$) is $\gamma$.
In particular, $\C_\Gamma$ has an orientation induced by $\gamma$. 
Let $C = C(\T^2)$ be the free abelian group generated by the set of (oriented)
one-cycles in $\T^2$. The standard basis in $\RR^2$ induces an isomorphism
$H_1(\T^2) \simeq \ZZ^2$. Let $h\colon C \rightarrow \ZZ^2$ be the homology map
in this basis, and let $\hat h = M^{-1}\circ h$ where $M$ is as in \eqref{eqn:NormalVector}.
We have that
\[
h(\C_\Gamma) = \gamma
\quad \text{and} \quad
\hat h(\C_\Gamma) = \Gamma.
\]
In the dimer literature it is more common to use $\hat h$ that $h$, see, e.g., \cite{HV07, KO06, UY11}. 

Let $\H$ be an oriented hyperplane arrangement (i.e., line arrangement) in $\T^2$. 
Viewing $\H$ as a union of lines, write $\H = \bigsqcup_{i=1}^m \H_m$
where two lines in $\H$ belongs to the same set $\H_k$ if and only if they are parallel.
We will say that $\H$ is a \emph{dual} arrangement
of the polygon $\N$ if there is a bijective relation between the set $\{\H_k \, | \, k=1, \dots, m\}$
the set of facets $\Gamma_k$, $k=1, \dots m$, of $\N$ given by $\hat h(\H_k) = \Gamma_k$.
The \emph{shell} of the coamoeba $\C$ is defined as the oriented hyperplane arrangement
\[
\H = \bigcup_{\Gamma \prec \,\N}\C_\Gamma,
\]
where the sum runs over all proper faces of $\N$. 
We can view $\H$ as an oriented hyperplane arrangement in $\T^2$ and,
by construction, the shell is a dual arrangement of $\N$.
We note that there exist dual arrangements of $\N$ which cannot be realized as the shell of some 
bivariate polynomial with Newton polygon $\N$.

That $\H$ captures topological properties of $\C$ can intuitively be seen from the
fact, shown in \cite{Joh13} 
(see also \cite{NS13a}, where $\H$ is called the \emph{phase limit set}), that 
\[
\overline \C = \C \, \cup \, \H.
\]

A hyperplane arrangement $\H$ is said to be \emph{simple} if any triple of distinct hyperplanes 
in $\H$ has empty intersection. It is not hard to show that the set of all polynomials $f\in (\CC^\times)^A$
whose shell is simple is open; its complement is a proper semi-analytic variety.
We will say that $f$, or $\H$, is generic if $\H$ is a simple hyperplane arrangement.
If $A$ is sparse along edges 
(that is, if for each facet $\Gamma$ the intersection $A\cap\Gamma$ is a dupleton)
then the space of polynomials whose shell is non-simple 
is the inverse image of the argument map of a hyperplane arrangement in $\T^A$.

\section{Graphs constructed from the shell $\H$.}
\label{sec:graphs}

We will in this section introduce our generalization of the Hanany--Vegh algorithm
and make a few important remarks regarding Yang--Baxter modifications.

\subsection{Johansson's index map}
By a construction of Johansson, see \cite[\S 6]{Joh13}, the complement of the
shell $\H$ can be equipped with an \emph{index map} $\iota$. That is, there is a map
\[
\iota\colon \pi_0\left(\T^2\setminus\H\right) \rightarrow \ZZ,
\]
where, for a generic $\theta \in P$, the magnitude $|\iota(P)|$ is a lower bound on, 
and has the same parity as, 
the number of points in the fiber $V\,\cap\, \Arg^{-1}(\theta)$. 

The index map $\iota$ is subject to the following \emph{crossing rule};
crossing a hyperplane of $\H$ (in the universal cover $\RR^2$ of $\T^2$) with tangent vector $\gamma$, 
along a smooth path with tangent vector $\ell$ at the point of intersection with $\H$, 
the index $\iota(P)$ increases or decreases by one depending on wether the pair 
$(\ell, \gamma)$ is a positively or a negatively oriented basis.
See Figure~\ref{fig:indexmap}, where a generic intersection point of $\H$
is illustrated.
It is clear that the crossing rule determines the indices $\iota(P)$ up to a universal shift.
To avoid confusion, we note that the index map $\iota$ is not a height function of the type commonly
appearing in the dimer literature.

\subsection{The odd index graph algorithm}
Assume that $\H$ is generic.
We will construct a pair of dual mixed graphs from the pair $(\H, \iota)$. 
In these mixed graphs,
each vertex has an assigned binary vertex weight (or color).
However, adjacent vertices differ in color only if the common edge is undirected
and, hence, the graphs are not colored in the strict graph theoretical meaning of the word.

\begin{figure}
\includegraphics[width=30mm]{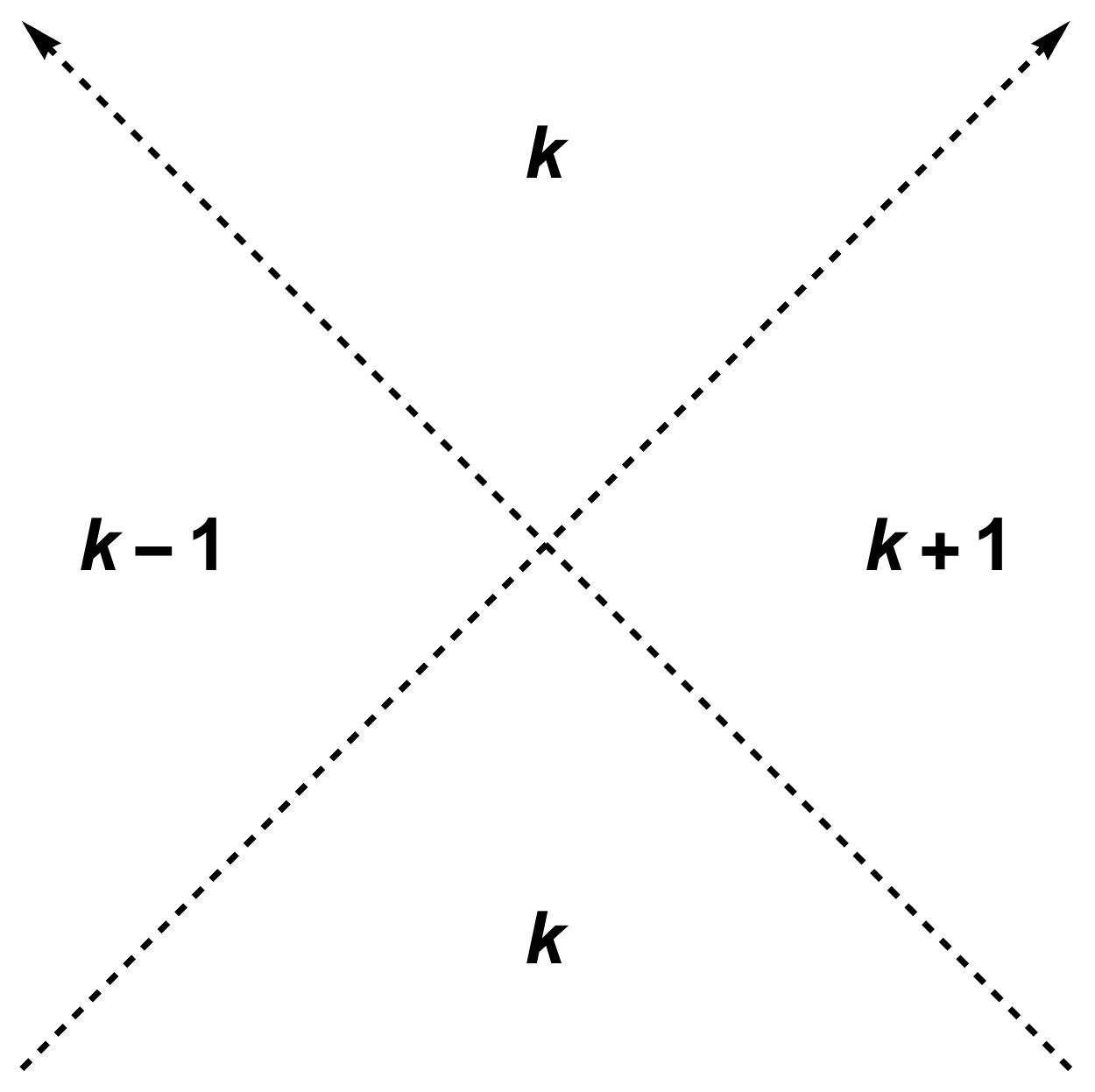}
\hskip8mm
\includegraphics[width=30mm]{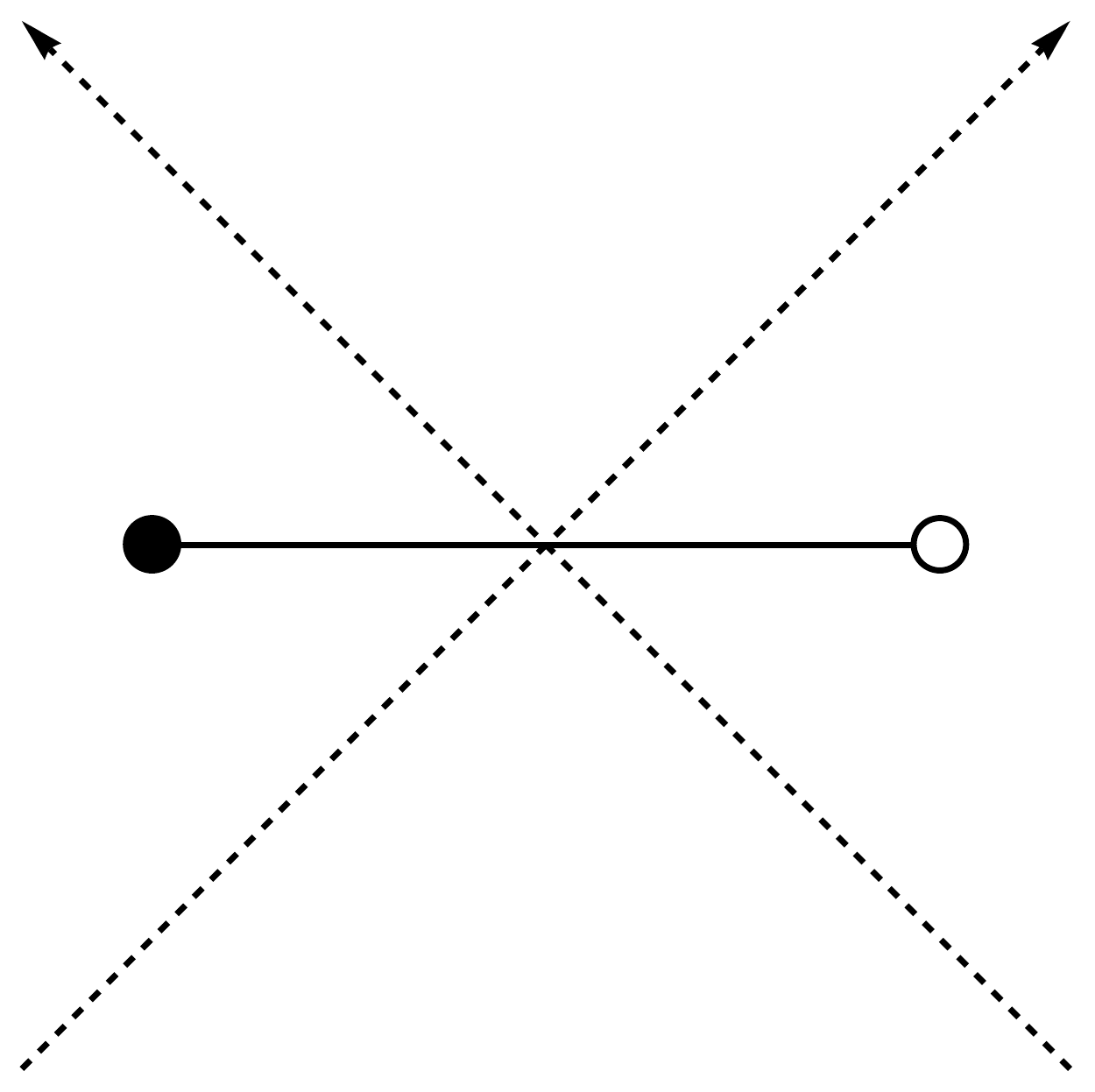}
\hskip8mm
\includegraphics[width=30mm]{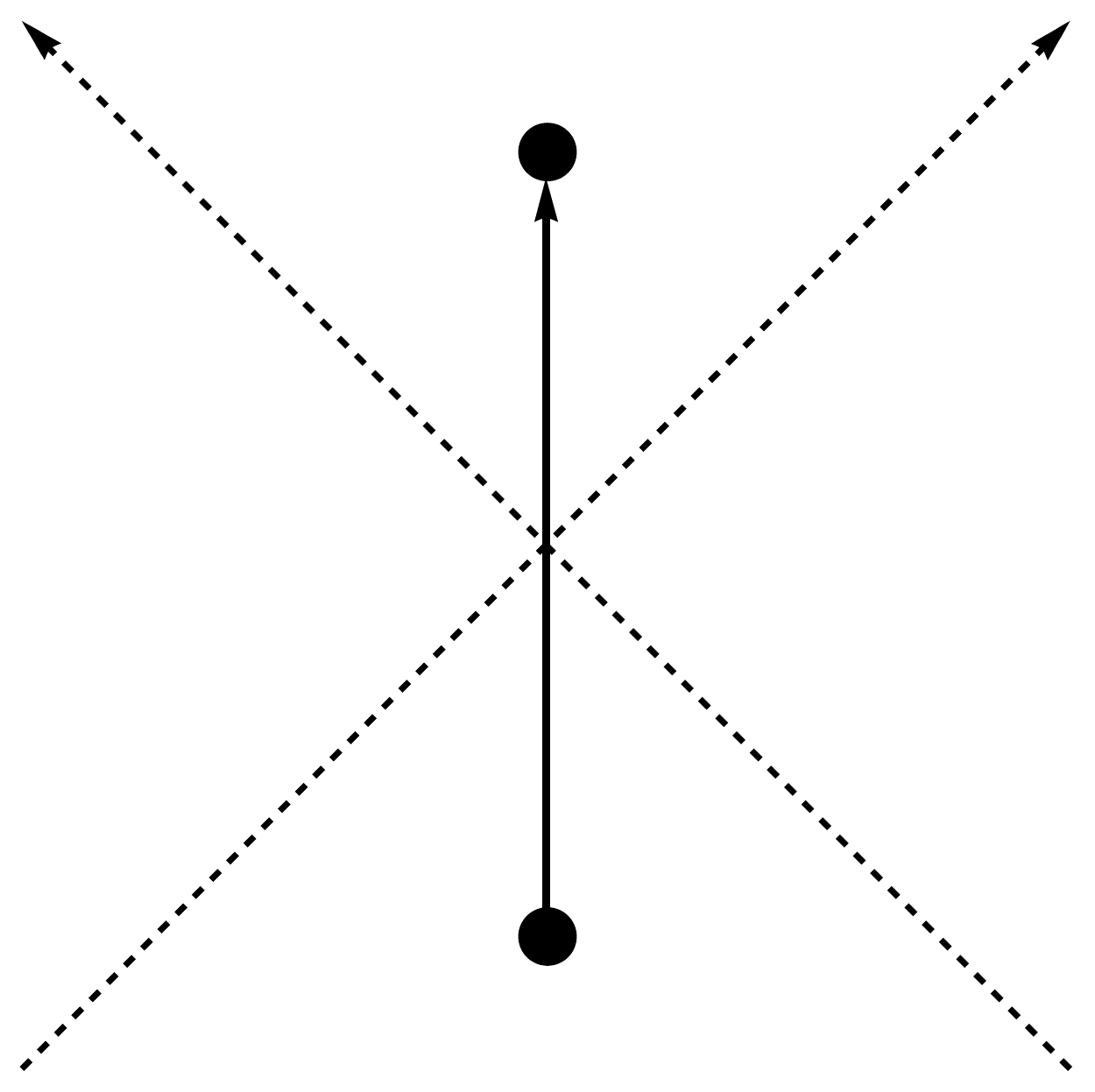}
\caption{
A generic intersection point of the shell $\H$ with the indices of Johansson's index map $\iota$,
the undirected edge of the dimer model, and the directed edge of the quiver.
}
\label{fig:indexmap}
\end{figure}

Before describing the algorithm, let us have a brief look at a generic intersection point $p$
of two oriented hyperplanes of the shell $\H$, as seen in Figure~\ref{fig:indexmap}.
Two of the four adjacent polygons, say $P_1$ and $P_2$, fulfill that $\iota(P_1) = \iota(P_2)$.
For one of these two polygons, say $P_1$, its edges are oriented towards $p$;
for the second polygon, $P_2$, its edges are oriented outwards from $p$.
That is, the shell $\H$ determines a natural orientation from $P_1$ to $P_2$,
as seen in the rightmost picture.

\begin{algorithm}[The index graph algorithm] 
\label{alg:IndexGraph}
$ $\\
\textbf{Input:} A pair $(\H, \iota)$ of a simple shell and its associated index map $\iota$.\\
\textbf{Output:} A mixed bicolored graph $G^- = (W,B,U,D)$, where $W$ and $B$ are the sets of white respectively black vertices, and $U$ and $D$ are the sets of undirected respectively directed edges.
\begin{enumerate}[1:]
\item \textbf{for} each $P\in \pi_0\left(\T^2\setminus\H\right)$ such that $\iota(P) \equiv 1$ modulo $4$ \textbf{do}
\item  add a white vertex $v = v(P)\in W$ 
\item \textbf{end for}.
\item \textbf{for}
For each $P\in \pi_0\left(\T^2\setminus\H\right)$ such that $\iota(P) \equiv -1$ modulo $4$ \textbf{do}
\item add a black vertex $v = v(P)\in B$
\item \textbf{end for}
\item \textbf{for} each intersection point $p \in P_1\cap P_2$ of polygons $P_1, P_2 \in \pi_0\left(\T^2\setminus\H\right)$ with $\iota(P_1) \equiv \iota(P_2) \equiv 1$ modulo $2$ \textbf{do}
\item $\quad$ \textbf{if} $\iota(P_1) \equiv \iota(P_2)$ modulo $4$  (with the orientation induced
by $\H$ going from $P_1$ to $P_2$) \textbf{then}
\item $\quad$ add a directed edge $e(P_1, P_2) = (v(P_1), v(P_2)) \in D$
\item $\quad$ \textbf{else}
\item $\quad$ add an undirected edge $e(P_1, P_2) = (v(P_1), v(P_2)) \in U$
\item $\quad$ \textbf{end if}
\item \textbf{end for}
\end{enumerate}
\end{algorithm}
We call the graph $G^-$ the \emph{index graph} associated with $(\H, \iota)$. 
Shifting all congruences by one, we obtain \emph{the even index graph algorithm}, 
whose output $G^+$ is called the \emph{even index graph} associated with the pair $(\H, \iota)$.

\begin{remark}
\label{rem:embedding}
The graphs $G^-$ and $G^+$ have a natural embedding into the torus $\T^2$,
provided that we allow for edges to be embedded as piecewise smooth curves;
map the vertex $v(P)$ to the center of mass of the polygon $P$, and
map the edge $e(P_1, P_2)$ to the union of the line segments from the centers of mass
of the polygons $P_1$ and $P_2$ to the corresponding intersection point $p$ of $P_1$
and $P_2$. In the case that $G^-$ is bipartite, it is known that
this embedding is isoradial, see \cite{UY11}.
\end{remark}

The underlying graphs (i.e., the graphs obtained by forgetting both the coloring and the orientations
of the directed edges) of $G^-$ and $G^+$ are dual as graphs embedded in $\T^2$.
In the case that $G^-$ is a bipartite graph, whose edges all are undirected, the graph $G^+$
is the dual quiver. In this case, the duality also respects the directions of edges;
the dual edge of an edge in the bipartite graph will be directed so that the black vertex lies on its left.
In the general case, the direction of the dual edge depends non-trivially on the indices $\iota(P)$.
Hence, any formulation the relation between $G^-$ and $G^+$ as a duality of
mixed graphs must contain all information encoded by the index map $\iota$. 
Such a formulation is not necessary for our purposes; we will work directly with $\iota$.

That either $G^-$ or $G^+$ is a bipartite graph is equivalent to that $\H$ is an \emph{admissible}
hyperplane arrangement, see \cite{UY11, UY13}. By definition, $\H$ (viewed as a polyhedral cell complex
with oriented edges) is said to be admissible if each edge bounds an oriented region.

\begin{figure}
\includegraphics[width=30mm]{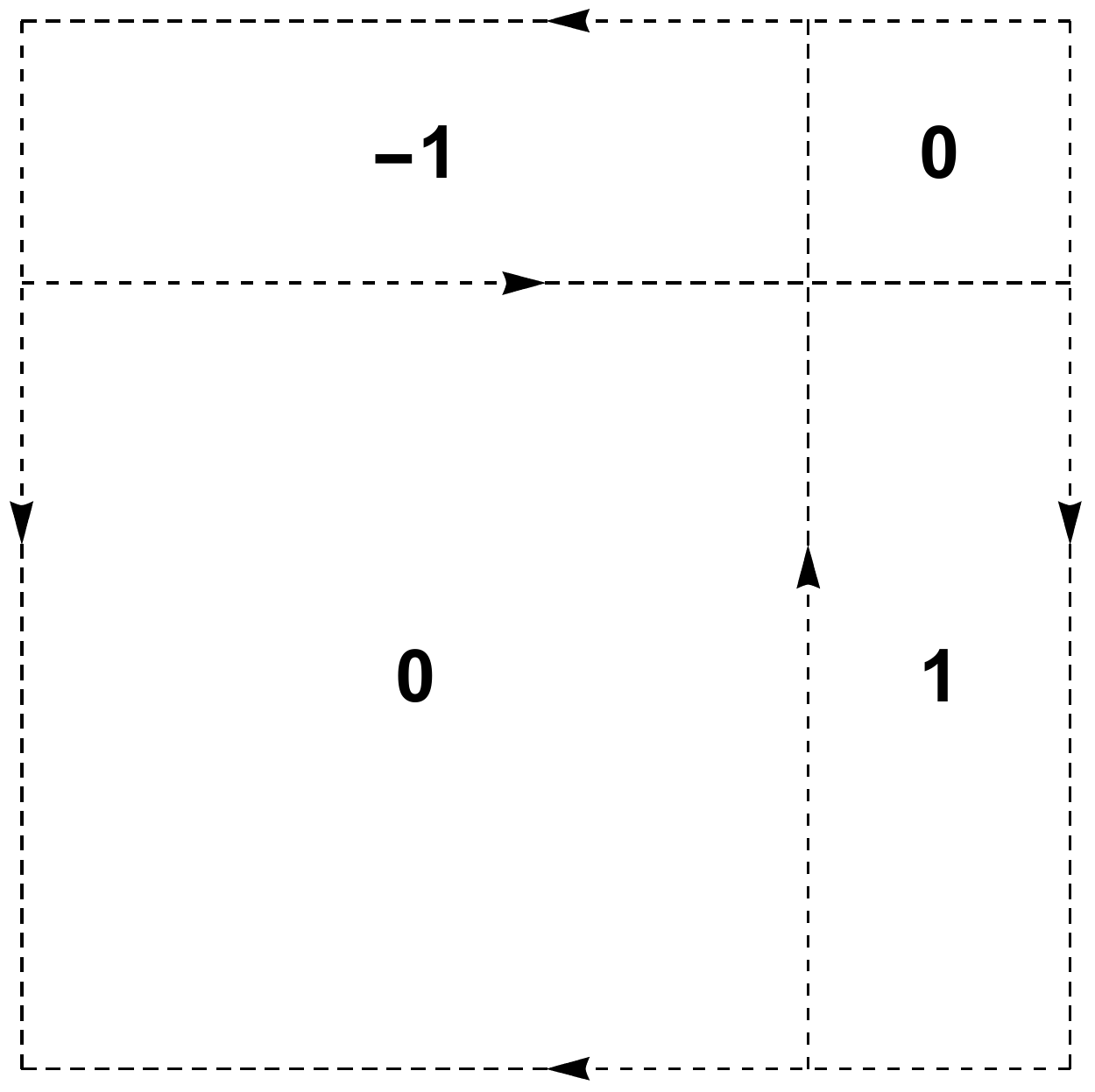}
\hskip8mm
\includegraphics[width=30mm]{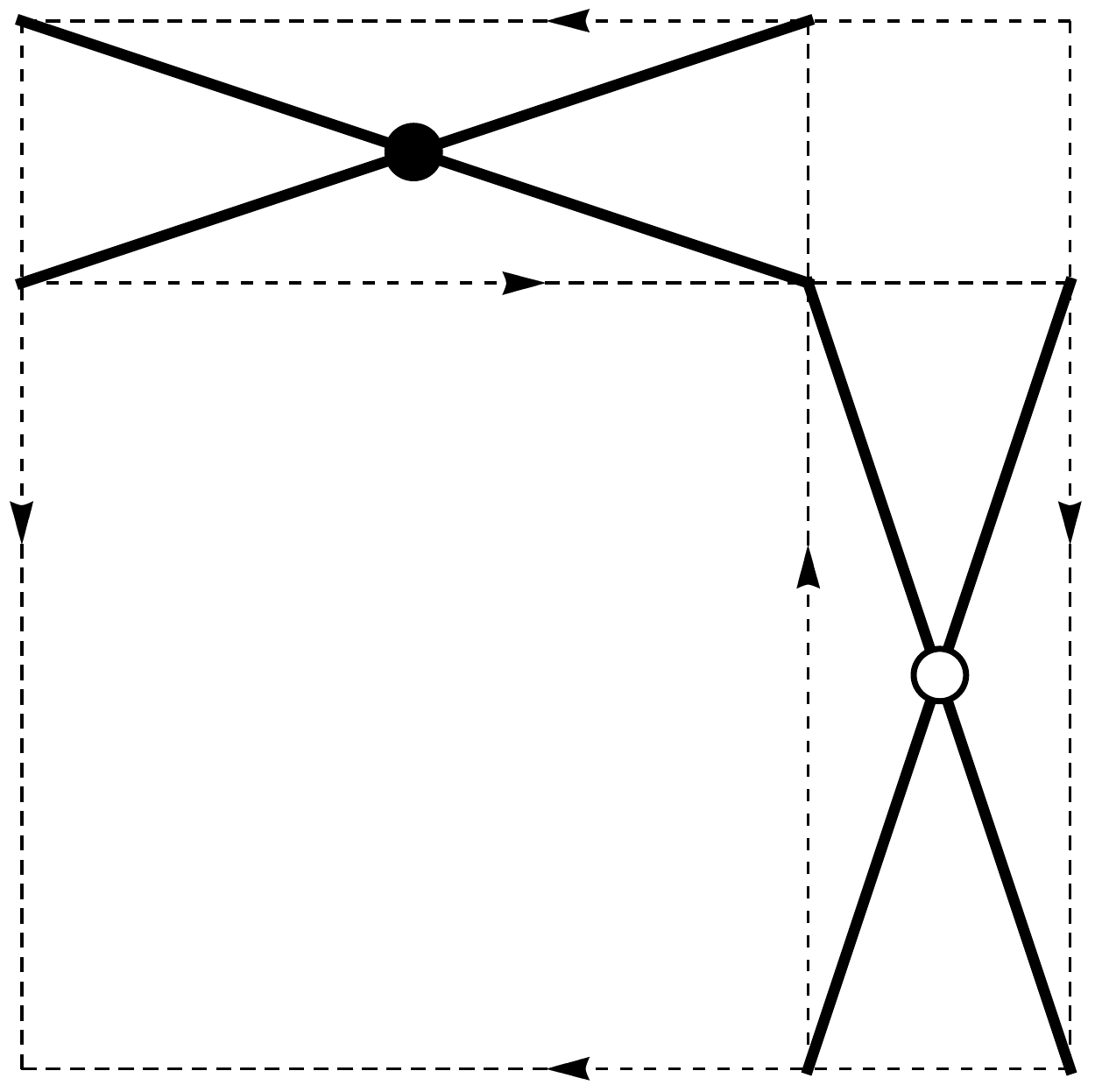}
\hskip8mm
\includegraphics[width=30mm]{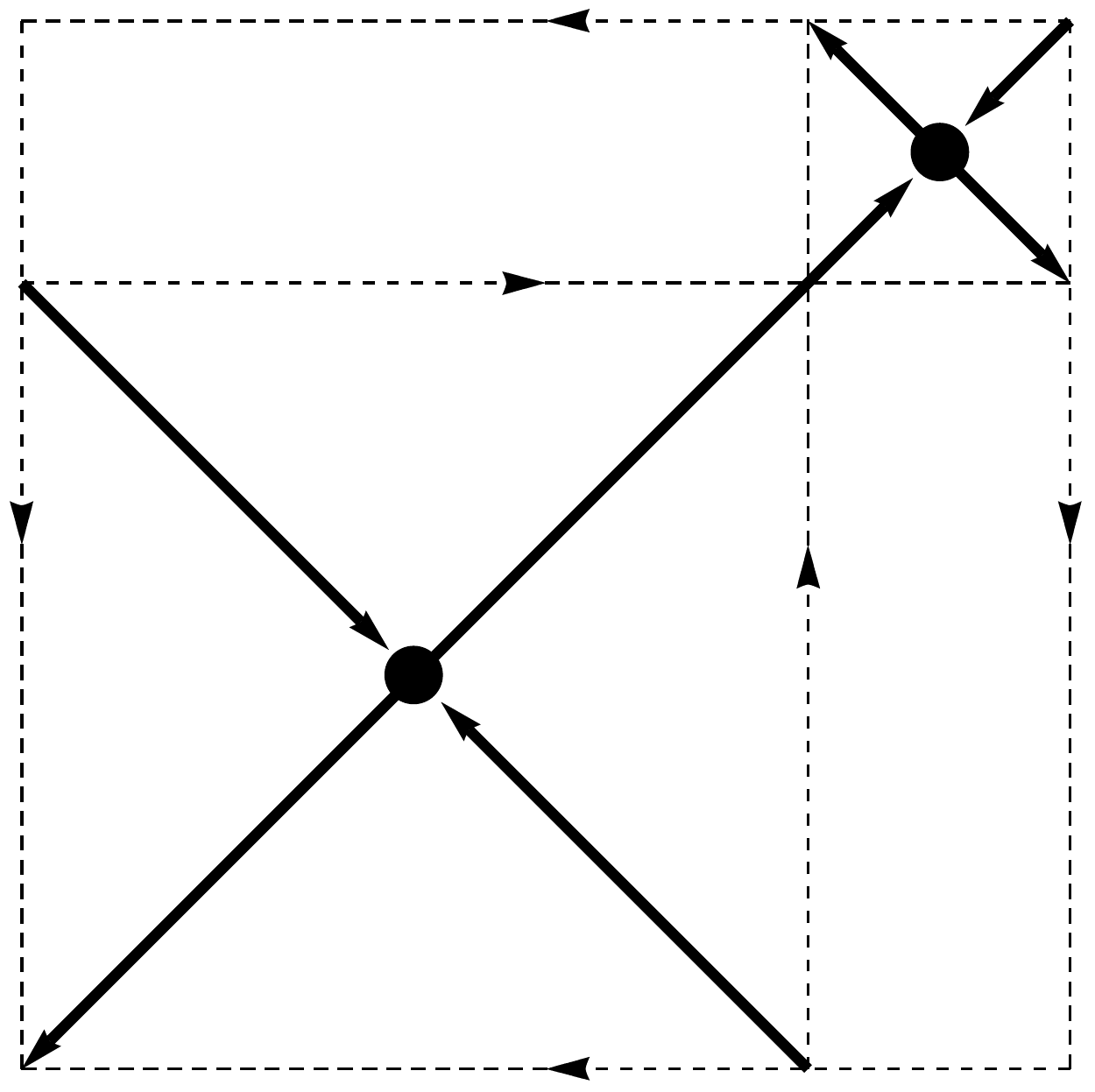}
\caption{
The odd and even index graphs defined by the shell $\H$ of the polynomial $f(z,w)$
from Example~\ref{ex:SquareEx}.
}
\label{fig:SquareEx}
\end{figure}

\begin{example}
\label{ex:SquareEx}
Let us consider the polynomial
\[
f(z,w) = 1 + z + w + i z w.
\]
The shell $\H$ and the indices $\iota(P)$ can be seen in the leftmost picture in 
Figure~\ref{fig:SquareEx}. In this case, the index graph $G^-$ is a bipartite graph and
$G^+$ is its dual quiver.
\end{example}

\subsection{Yang--Baxter modifications}
\label{ssec:YangBaxter}
Let us consider how a small modification of the hyperplane arrangement $\H$ locally acts on
the mixed graph $G^-$ obtained through the odd index graph algorithm. 
These actions are known in the physics literature as 
\emph{Yang--Baxter modifications} \cite{HV07}. It is not hard to show that, up to graph isomorphisms, 
there are two distinct modifications of the mixed graphs $G^-$ and $G^+$, shown
in Figure \ref {fig:Mod}. Note, however, that if we require that $\H$ is a hyperplane arrangement, 
then not all Yang--Baxter modifications of
the graph $G^-$ can be realized by perturbations of the hyperplanes in $\H$. 
 
\begin{figure}
\begin{tabular}{ccccccc}
\raisebox{-.5\height}{\includegraphics[width=15mm]{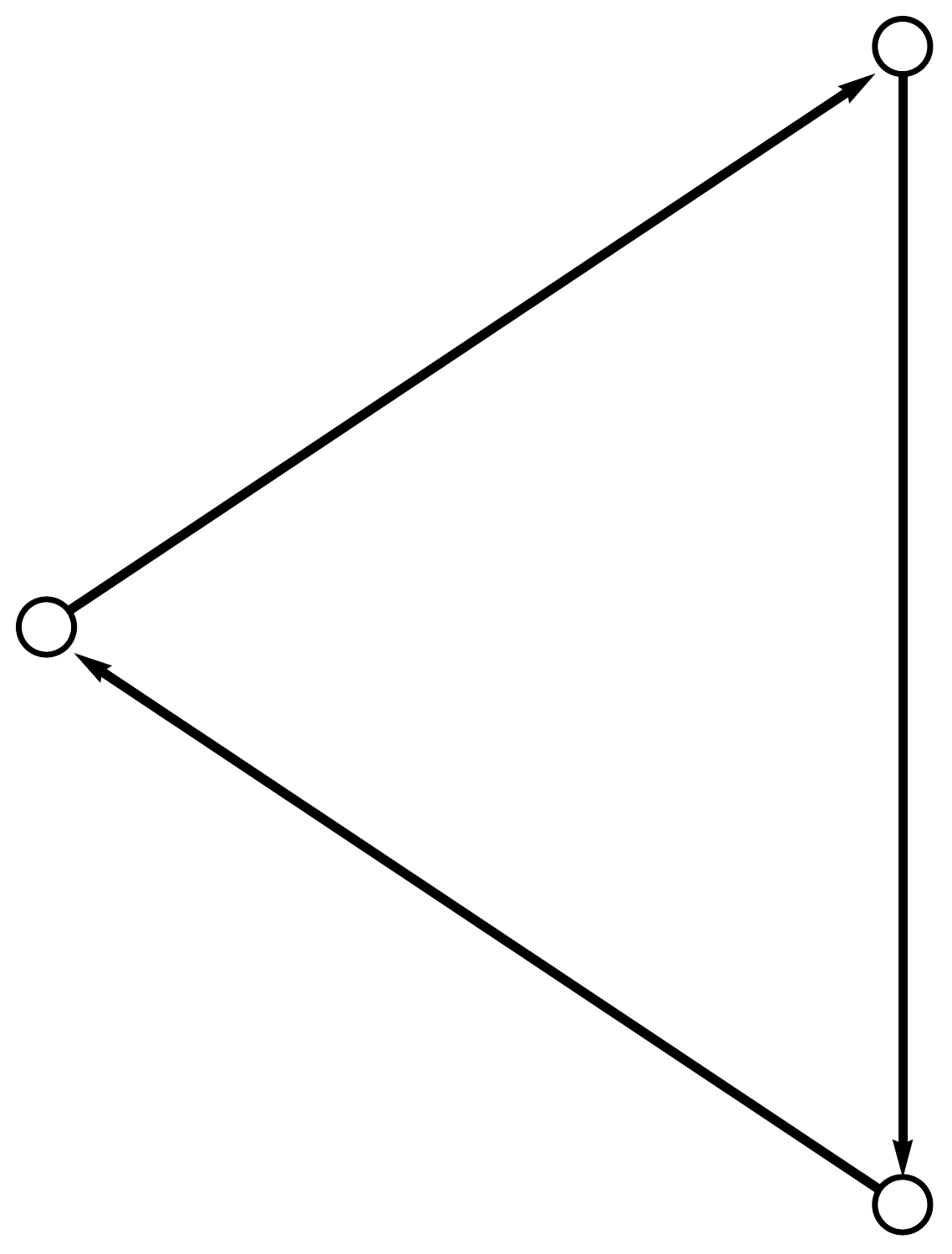}}
&
\hskip2mm
$\longleftrightarrow$
&
\hskip2mm
\raisebox{-.5\height}{\includegraphics[width=15mm]{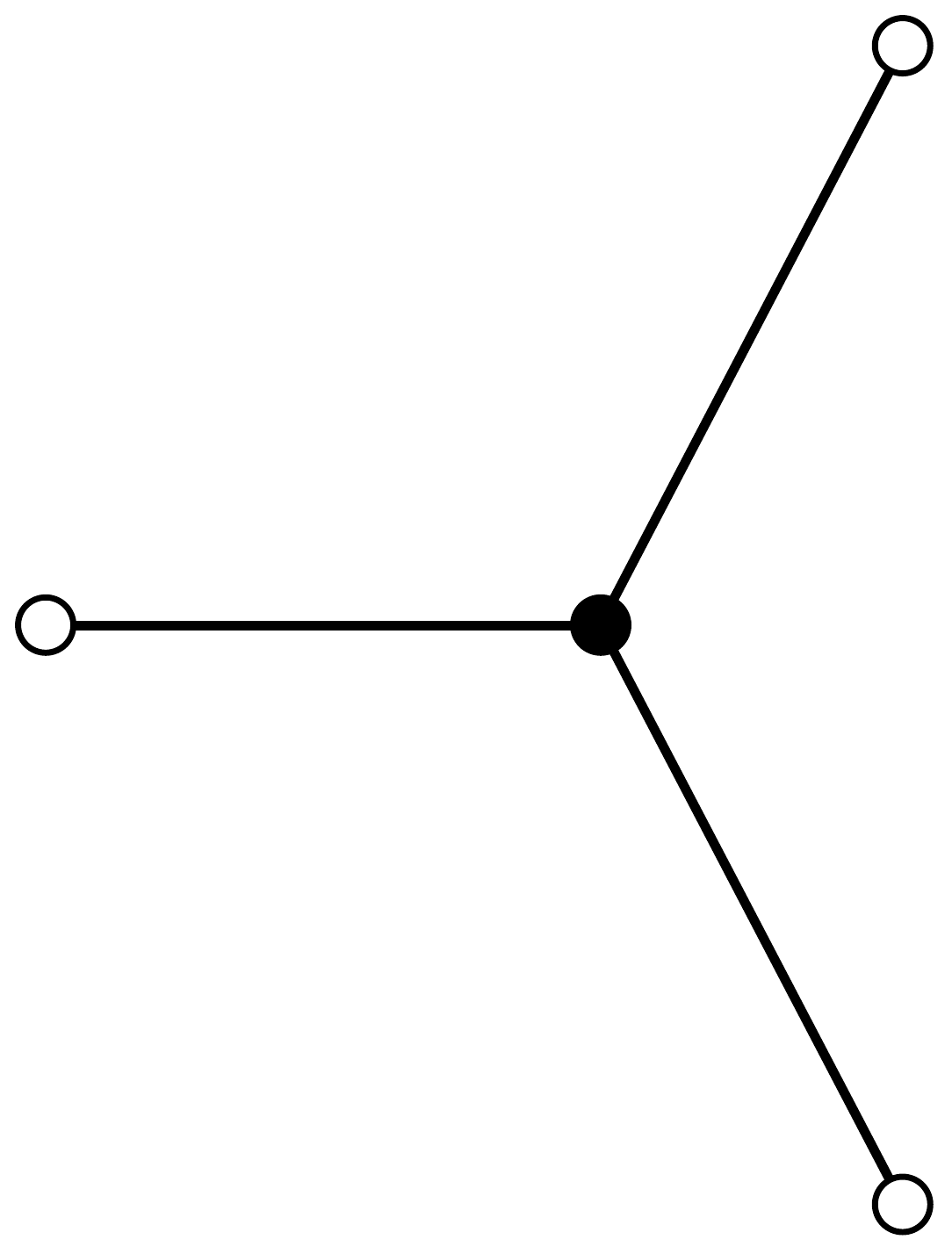}}
&
\hskip12mm
$ $
&
\raisebox{-.5\height}{\includegraphics[width=15mm]{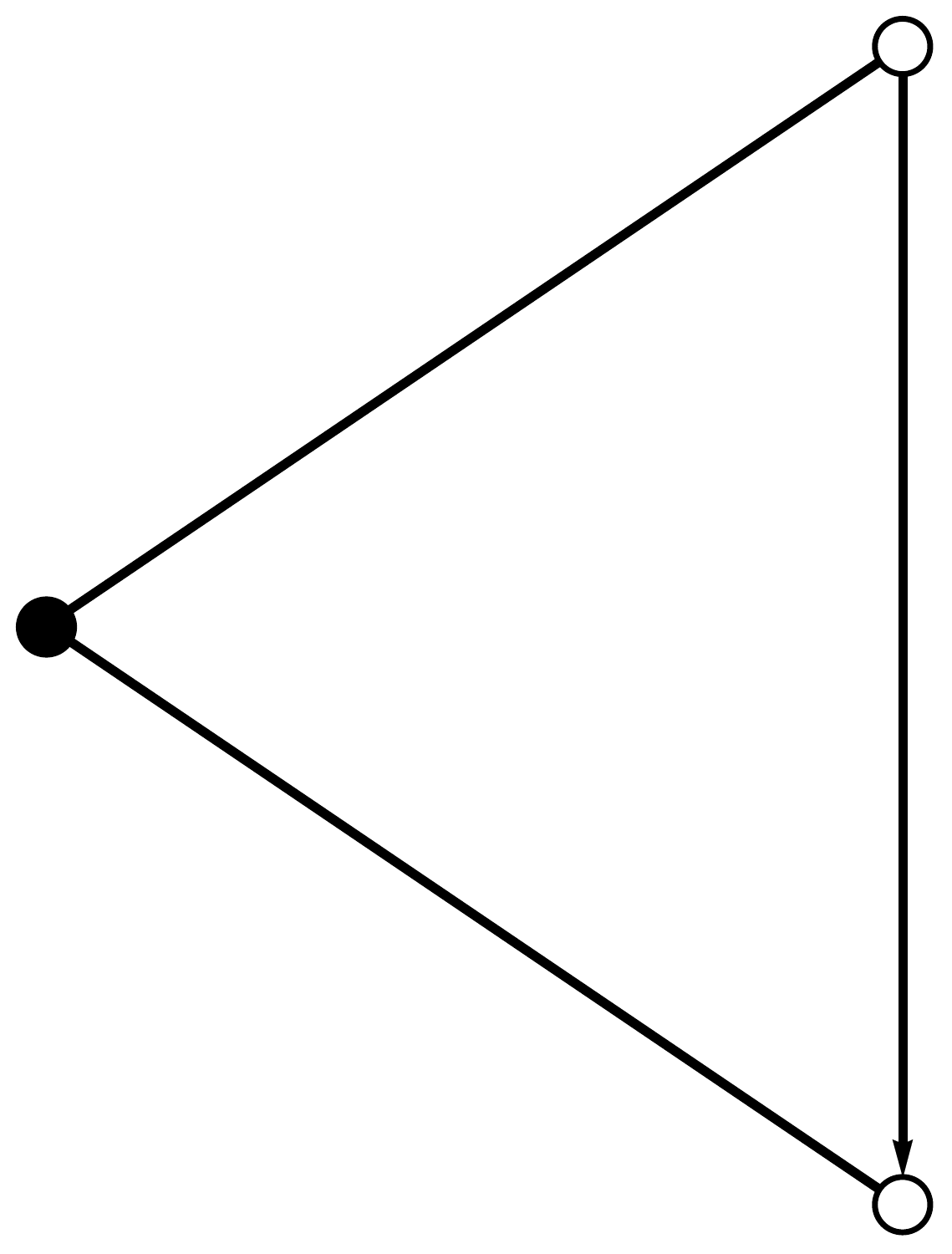}}
&
\hskip2mm
$\longleftrightarrow$
&
\hskip2mm
\raisebox{-.5\height}{\includegraphics[width=15mm]{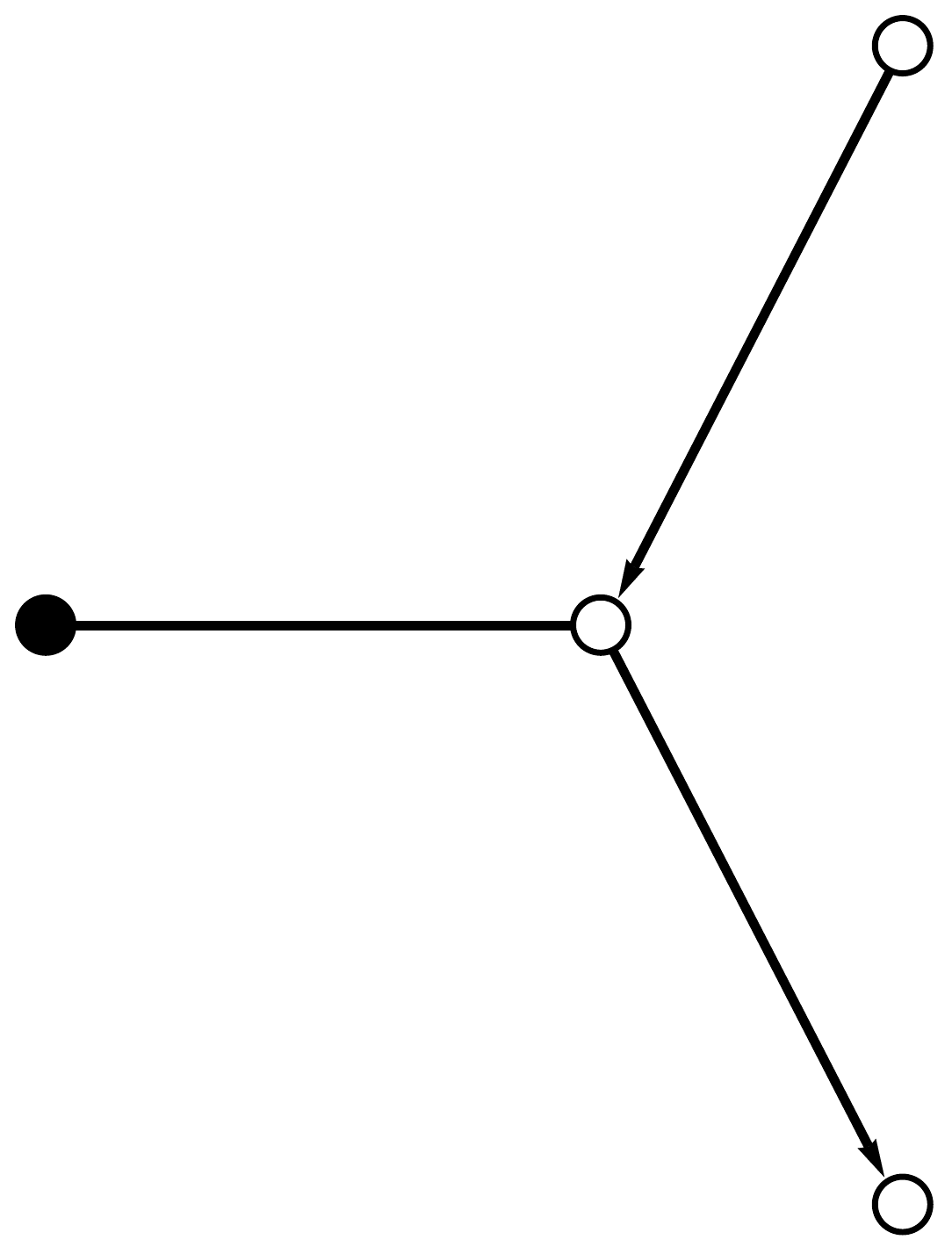}}
\\

\\
\raisebox{-.5\height}{\includegraphics[width=15mm]{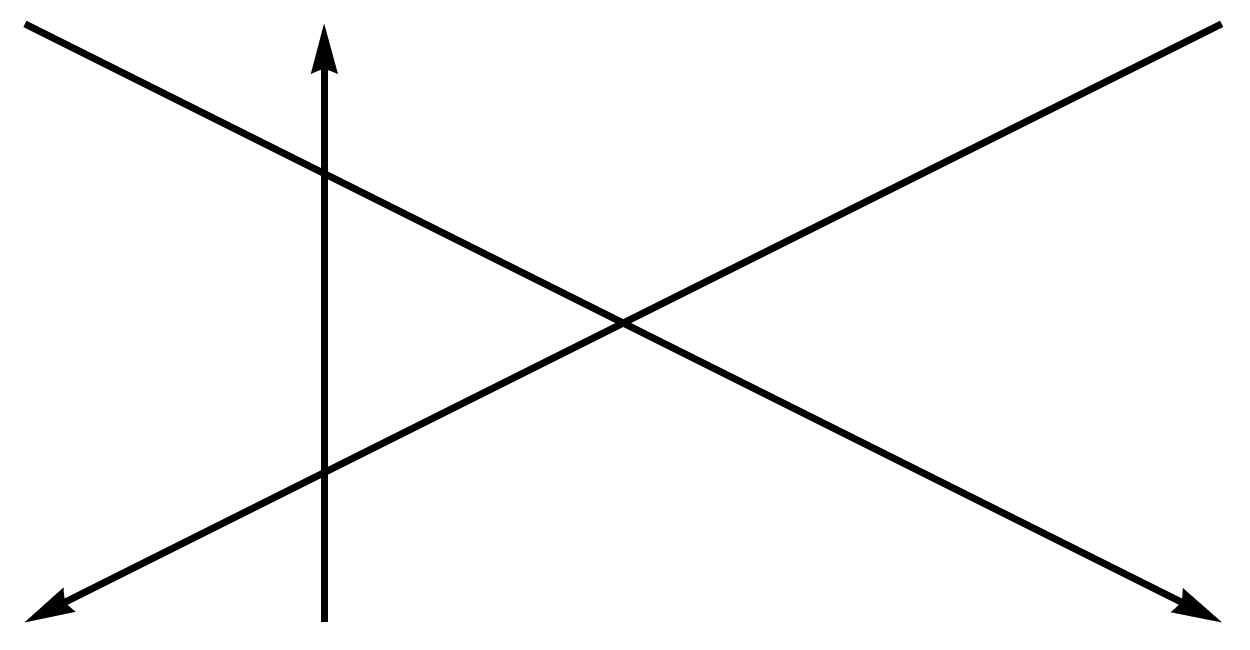}}
&
\hskip2mm
$\longleftrightarrow$
&
\hskip2mm
\raisebox{-.5\height}{\includegraphics[width=15mm]{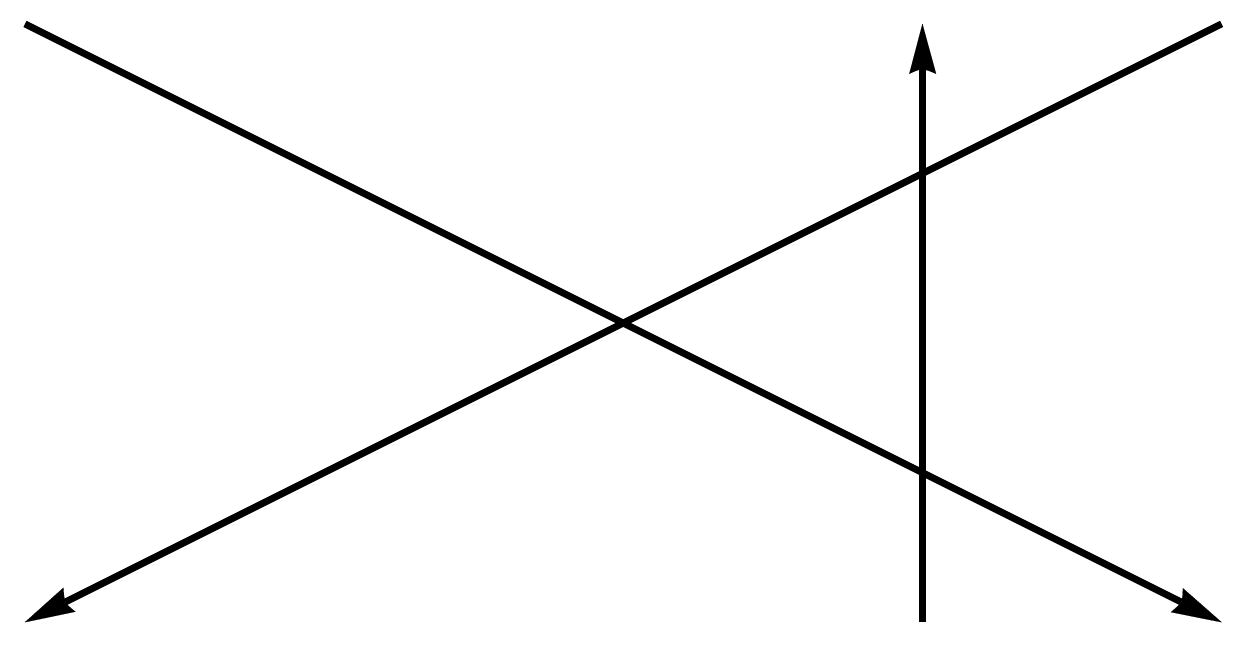}}
&
\hskip12mm
$ $
&
\raisebox{-.5\height}{\includegraphics[width=15mm]{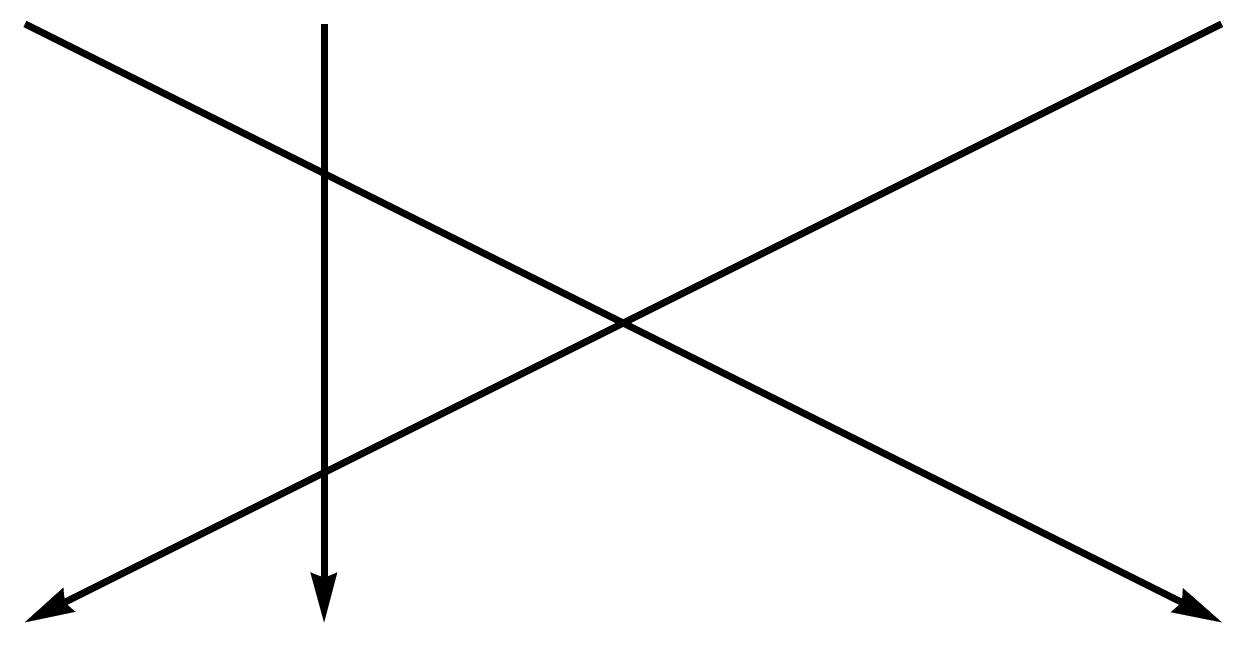}}
&
\hskip2mm
$\longleftrightarrow$
&
\hskip2mm
\raisebox{-.5\height}{\includegraphics[width=15mm]{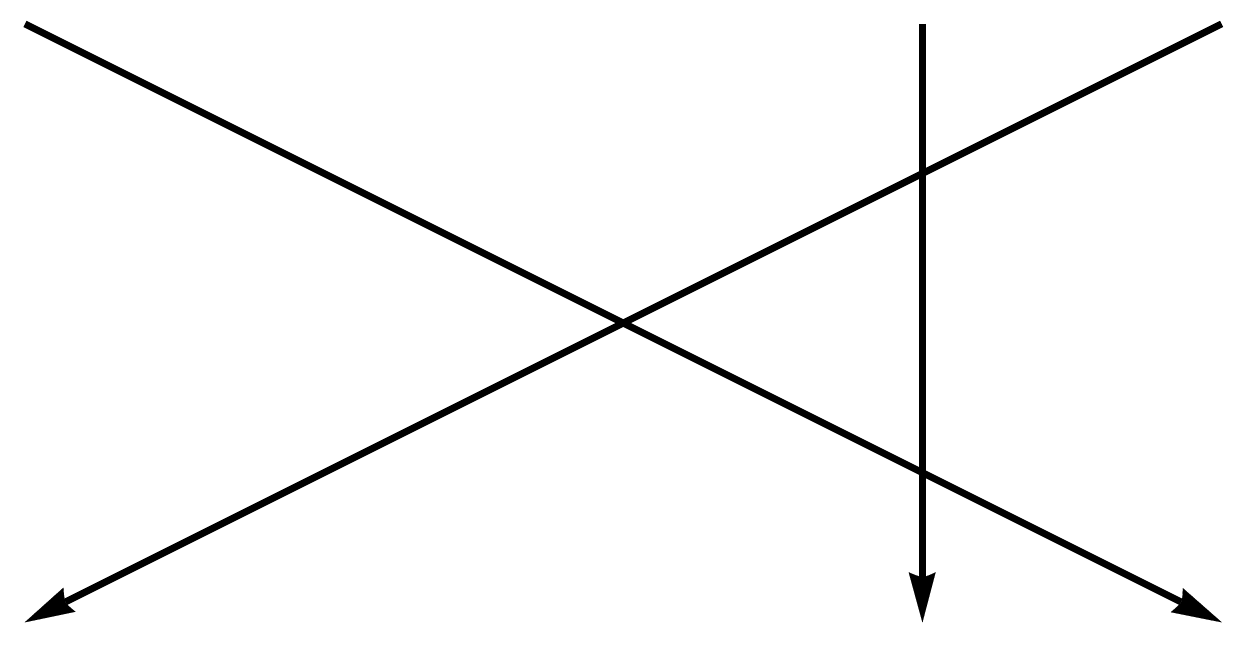}}
\end{tabular}
\caption{Above: the first and second Yang--Baxter modifications. Below: the corresponding local deformations of 
the hyperplane arrangement $\H$.}
\label{fig:Mod}
\end{figure}

We only need the first Yang--Baxter modification to prove our main theorem. 
This modification applies in the following situation. Assume that the arrangement
$\H$ has a cell $P \in \pi_o(\T^2\setminus \H)$, of even index, 
which is a triangle with oriented boundary.
The action on the graph is induced by locally deforming one of the 
hyperplanes in $\H$, forming instead a triangle with non-oriented boundary.
The cell $P$ is by this deformation replaced by a cell $Q$ of odd
index. The corresponding modification of $G^-$ can be seen in the uppermost
left picture in Figure \ref{fig:Mod}.

\begin{remark}
The algorithm of Hanany--Vegh \cite{HV07} depends on the one is provided an admissible
arrangement $\H$. The algorithm of Stienstra \cite{Sti08} takes an arbitrary simple arrangement
as input, and searches along a tree of Yang--Baxter modifications for a simple arrangement.
Algorithm~\ref{alg:IndexGraph} is equivalent to Stienstra's algorithm, which the crucial difference
that we perform the steps in a different order. We construct, from the arrangement $\H$ a mized graph
$G$, and then we perform Yang--Baxter modifications. In examples, it suffices to
perform Yang--Baxter modifications which reduce the number of directed edges. This greatly reduces
the number of trial-and-errors required.
\end{remark} 
 
 \begin{proposition}
 \label{pro:Consistent}
Let $\H$ be a dual hyperplane arrangement of the polygon $\N$, and let $\iota$
be its index map.
Assume that $|\iota(P)| \leq 2$ and that each polygon $P$ with $|\iota(P)| = 2$ 
is a triangle.
Then, the dimer model $G$ obtained from $\H$ by 
the index graph algorithm and Yang--Baxter modificantions
is a consistent dimer model with characteristic polygon $\N$.
 \end{proposition}
 
 \begin{proof}
As defined in \cite[Definition 3.5]{IU11},
a dimer model $G$ obtained from an admissible arrangement of piecewise linear
curves $\H$ is consistent if and only if
\begin{enumerate}[$a$)]
\item No closed piecewise linear curve $H\in \H$ has $\hat h(H) = 0$.
\item No closed piecewise linear curve $H$ has a self-intersection in the universal cover $\RR^2$ of $\T^2$.
\item No two closed piecewise linear curves $H_1$ and $H_2$ in $\H$ intersect in the universal cover $\RR^2$
of $\T^2$ in the same direction more than once.
\end{enumerate} 
If we generalize from hyperplanes to piecewise linear curves, then we can act on $G$ by applying
Yang--Baxter modifications locally.
We note that $|\iota(P)|$ is a local maximum then the boundary of $P$ is oriented. 
If, futhermore,
$P$ is a triangle, then applying the first Yang--Baxter modification leaves
the number of intersection points between any pair of lines in $\H$ invariant. Hence, the dimer model $G$
is consistent, as the properties $a$)--$c$) are fulfilled by any dual hyperplane arrangement $\H$ of the polygon $\N$.
Finally, that the characteristic polygon of the dimer model $G$ is equal to $\N$ follows from \cite{FHKV08}.
\end{proof}
 
\begin{remark}
It was shown in \cite{IU11} that the dimer model $G$ is consistent if and only if it is properly ordered
in the sense of Gulotta, and 
it was shown in \cite[Theorem 3.1]{Gul08} that if a dimer model is properly ordered,
then the number of two-dimensional faces is equal to $2\area(\N)$.
In particular, with $\H$ and $\N$ as in Proposition \ref{pro:Consistent},
the admissible arrangement of pieceswise smooth curves obtained from $\H$ by applying
Yang--Baxter modifications has $2\area(\N)$-many cells of index zero.
\end{remark}

\section{Intermezzo: Polygons with no admissible dual hyperplane arrangements}
If $\H$ is an admissible hyperplane arrangement, then the linear Hanany--Vegh algorithm
constructs a dimer model $G$ on the torus $\T^2$. 
We will now answer the question of whether one can always find an admissible hyperplane arrangement $\H$,
which has been raised on several occasions (see, e.g.,  \cite[Remark 6.10]{Sti08}),
in the negative.
Note that if we allow piecewise smooth curves then, 
by Gulotta's algorithm \cite{Gul08}, one can construct an admissible arrangement.
However, since the shell $\H$ is a hyperplane arrangement, such a generalization
is not suitable for our purposes.

\begin{figure}
\begin{tabular}{ccc}
\raisebox{-.5\height}{\includegraphics[width=30mm]{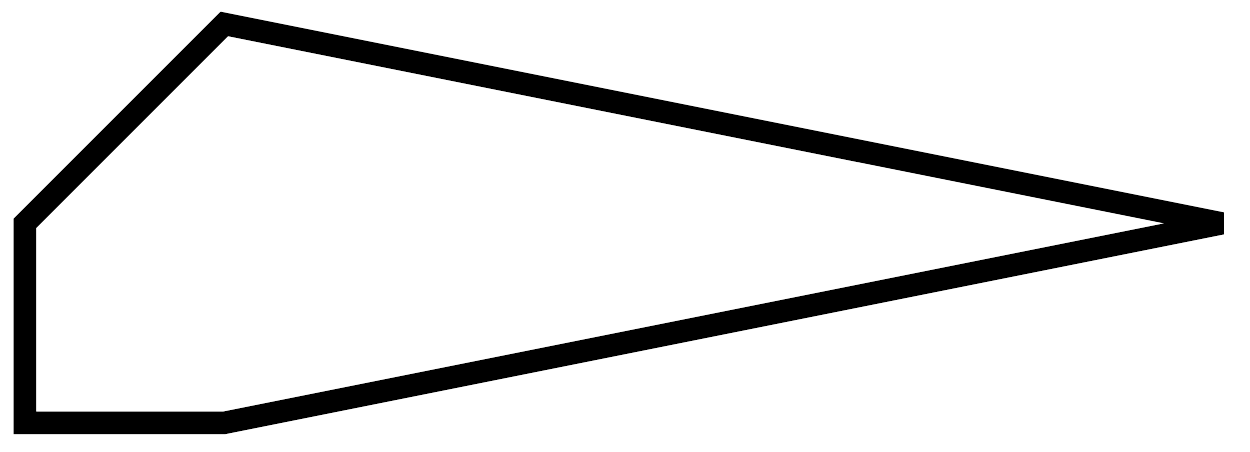}}
&
\hskip8mm
\hskip8mm
\raisebox{-.5\height}{\includegraphics[width=30mm]{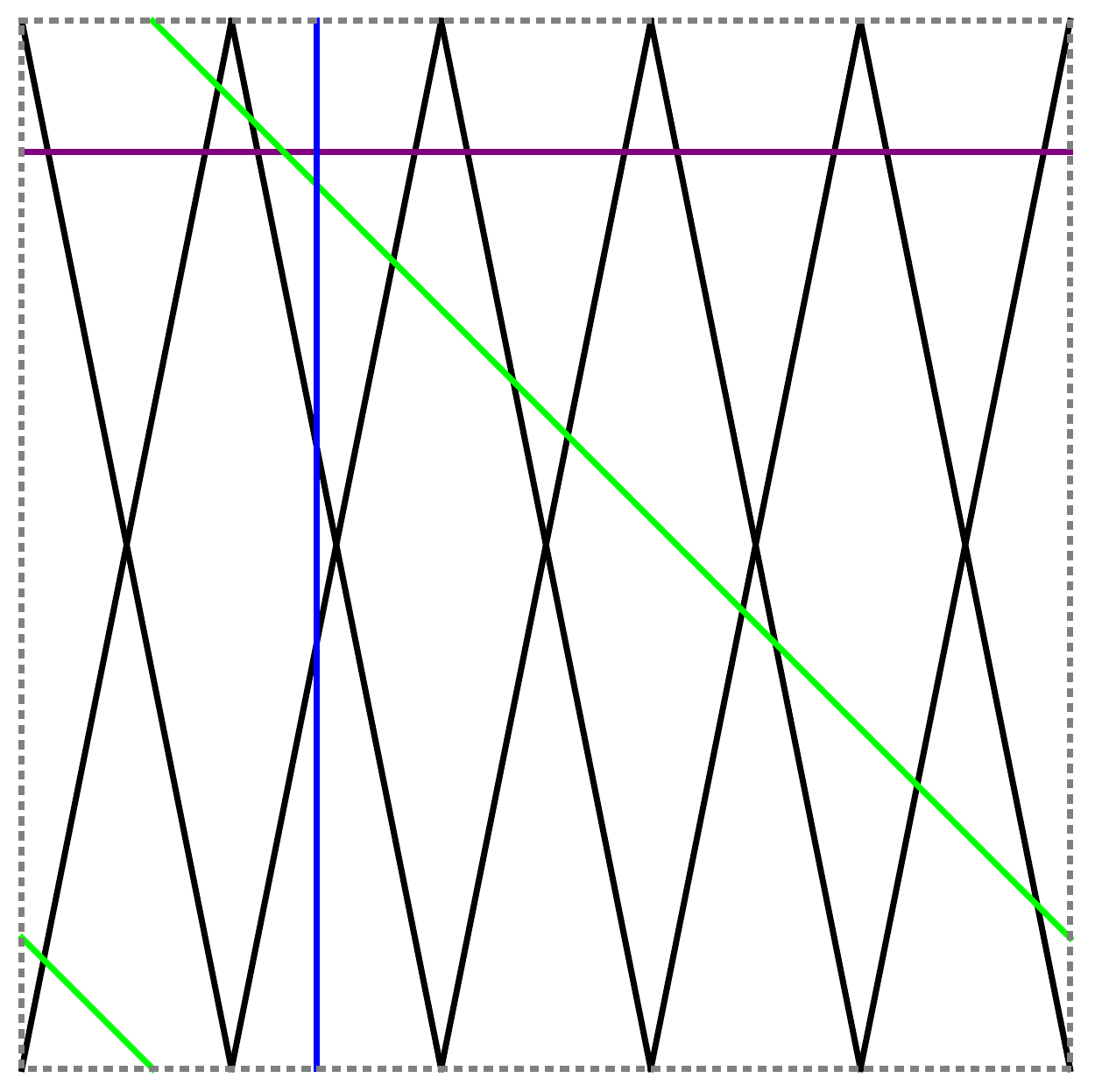}}
\end{tabular}
\caption{Left: the polygon from Example~\ref{ex:HananyVeghExample}.
Right: a non-admissible arrangement. The two tiers of the arrangement $\H_2$ is seen in black.
The hyperplane $H_1$, in green, intersects three of the line segments $\ell_j$ on the first tier, and three on the second tier.
One line segment $\ell_j$ does not intersect $\H_1$.}
\label{fig:NoAdmissibleEx}
\end{figure}
 
\begin{example}
\label{ex:HananyVeghExample}
Consider the polygon $\N$ with vertices $(0,0)$, $(1,0)$, $(0,1)$, $(k+1,1)$, and $(1,2)$,
for some $k\in \ZZ$ with $k > 1$. 
The edges of $\N$ are $\Gamma_1 = (-1,-1)$, $\Gamma_2= (0,-1)$, $\Gamma_3 = (1,0)$, $\Gamma_4 = (k, 1)$, 
and $\Gamma_5 = (-k, 1)$. Let $H_1, \dots, H_5$ denote the corresponding hyperplanes in $\T^2$,
and define $\H_1 = H_1\cup H_2\cup H_3$ and $\H_2 =  H_4 \cup H_5$,
so that, as a set, $\H = \H_1 \cup \H_2$.

The arrangement $\H_2$ is a Lozenge tiling of $\T^2$, see the black arrangement in Figure~\ref{fig:NoAdmissibleEx}.
It divides the hyperplane $H_5$ into $2k$-many parallel line segments $\ell_j$ for $j=1, \dots, 2k$. 
It is easy to see that if one of these line segments $\ell_j$
does not intersect $\H_1$, then it does not bound an oriented region of $\T^2\setminus\H$.
We make three remarks.
Firstly, the line segments $\ell_j$ can be divided into two \emph{tiers}; the upper tier consists of the $k$ segments
that intersect $H_2$, the bottom tier consists of the remaining $k$ segments,
see Figure~\ref{fig:NoAdmissibleEx}.
Secondly, the hyperplane $H_1$ intersects at most $\big\lceil\frac{k+1}2\big\rceil$-many
line segments on each tier. Thirdly, the hyperplane $H_3$ intersects exactly one of the line segments $\ell_j$.
Let
\[
m = \left\lceil\frac{k+1}2\right\rceil + 1.
\]
It follows that at most $m$ of the line segments on the bottom tier intersect $\H_1$. In particular,
if $m < k$, then there is no dual admissible hyperplane arrangement $\H$ of the polygon $\N$.
We can conclude that if $k \geq 5$, then $\N$ has no dual admissible hyperplane arrangement.
That is, if $k \geq 5$ then it is not possible to construct a dimer model whose characteristic
polygon is $\N$ using the Hanany--Vegh algorithm \cite{HV07} or Stienstra's algorithm \cite{Sti08}.

That a line segment $\ell_j$ does not intersect $\H_1$ is not the only possible obstruction for a hyperplane
arrangement $\H$ to be admissible. A computer aided inspection suggests that no 
admissible arrangements exist for $k \geq 3$. Admissible arrangements
exist in the cases $k=1$ and $k=2$.
\end{example}

\section{Dimer models and Coamoebas}
\label{sec:DimerCoamoebas}

In this section we will discuss the relation between coamoebas and dimer models.
The aim is to explain the observations from \cite{FU10, UY11, UY13} that 
the dimer model is a deformation retract of the coamoeba of the characteristic
polynomial. 
We need the following definition in order to atomize this observation.

\begin{definition}
Let $f\in (\CC^\times)^A$, and let $\H$ and $\iota$ be the shell and index map associated with the coamoeba $\C$. 
We define the
\emph{combinatorial coamoeba} $\D$ to be the set
\[
\D = \bigcup_{|\iota(P)| > 0} \overline{P},
\]
where the union is taken over all polygons $P \in \pi_0\left(\T^2\setminus \H\right)$ whose
index is of positive magnitude.
\end{definition}

It is a theorem of Johansson \cite[Theorem 5.1]{Joh13} that $\D \subset \overline{\C}$. Typically, this inclusion
is strict. Furthermore, \cite[Lemma 2.3]{FJ14} shows that there is an injective map
$\pi_0(\T^2\setminus \overline{\C}) \rightarrow \pi_0(\T^2\setminus \D)$ given by inclusion as subsets of $\T^2$.
In general, this map need not be surjective. 
Thus, the relation between $\overline{\C}$ and $\D$ is non-trivial.
Typically, in the case when  the dimer model $G$ is a deformation retract of the coamoeba $\overline{\C}$
two relations hold simultaneously. Firstly, $G$ is a deformation retract of $\D$ and, secondly,
it holds that $\pi_0(\T^2\setminus\D) = \pi_0(\T^2\setminus \overline{\C})$. We mention this as
$\D$, due to its combinatorial nature, is a more accessible object of study than $\overline{\C}$. 
In both cases, the cardinality of the $0$'th fundamental group $\pi_0$
is at most twice the area of the Newton polygon $\N$, see \cite{FJ15}.
To prove Theorem \ref{thm:DeformationRetract} we need the following preliminary
results regarding the combinatorial coamoeba $\D$ and the index map $\iota$.
We begin with an immediate consequence of
Algorithm~\ref{alg:IndexGraph}.

\begin{theorem}
\label{thm:P2GivesDeformationRetract}
Assume that $\H$ is such that\/ $|\iota(P)| \leq 2$ and for each $P$ with $|\iota(P)| = 2$
it holds that $P$ is a triangle. Then the dimer model $G$ obtained from $G^-$ using
Yang--Baxter modifications is a (strong) deformation retract of\/ $\D$.
\end{theorem}

\begin{figure}
\includegraphics[width=30mm]{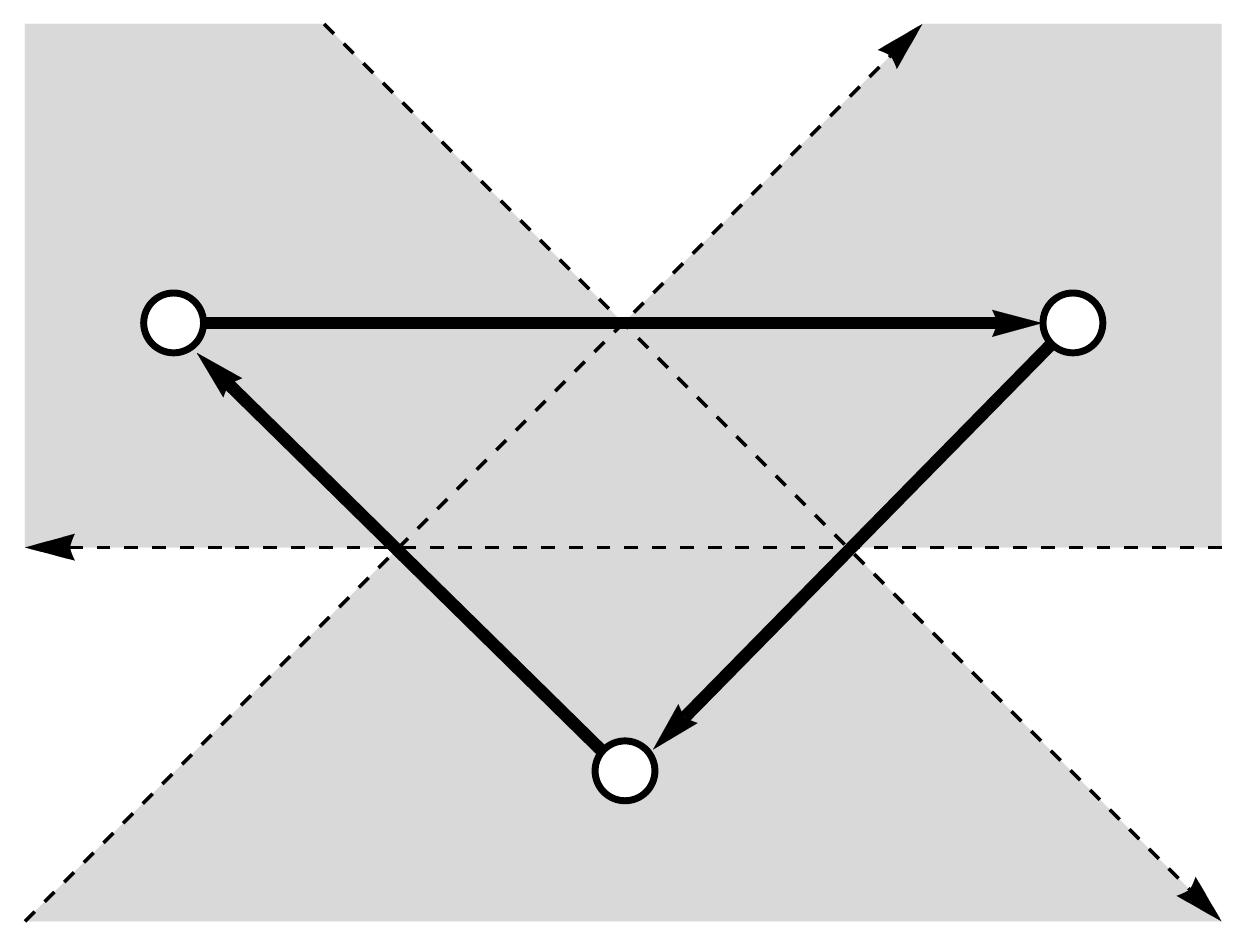}
\hskip8mm
\includegraphics[width=30mm]{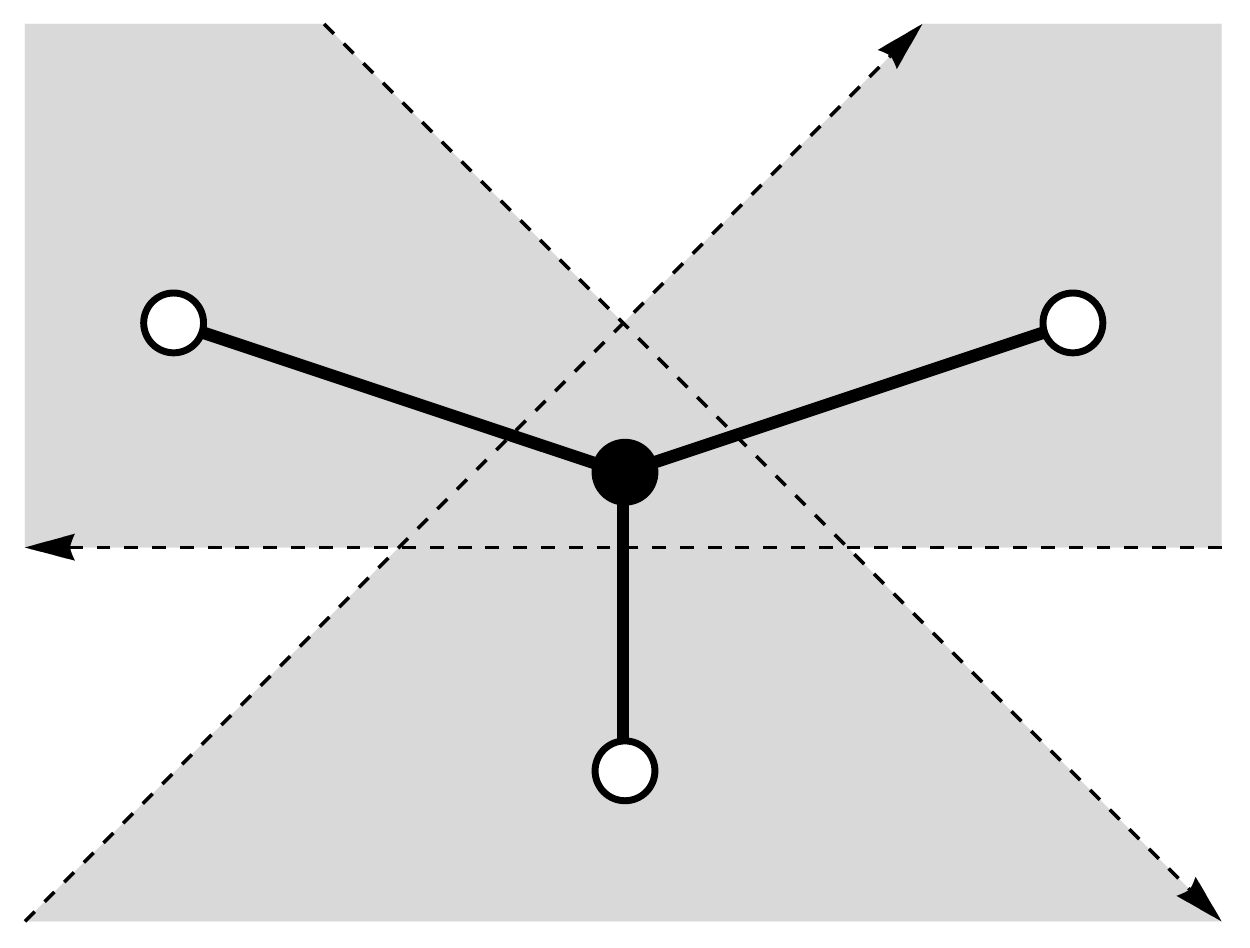}
\caption{
The graph $G^-$ and $\D(f)$, in the neighborhood of a triangular polygon $P$ of index two,
before and after applying the first Yang--Baxter modification.
}
\label{fig:P2GivesDeformationRetract}
\end{figure}

\begin{proof}
Notice that a cell $P$ with $|\iota(P)| = 2$ is, after applying the first Yang--Baxter modification,
associated to a vertex of the graph $G$, see Figure \ref{fig:P2GivesDeformationRetract}.
In particular, the dimer model $G$ 
admits the following embedding into $\T^2$. For each cell $P$ with non-vanishing index
we add a vertex $v(P)$. The vertex is white if $\iota(P) = 1$ or $-2$, 
and it is black if $\iota(P) = -1$ or $2$. 
For each cell $P$ with $|\iota(P)| = 2$ we add
embed the edges of $G$ containing $v(P)$ as straight line segments, see the rightmost picture in
Figure \ref{fig:P2GivesDeformationRetract}.
The remaining edges are added as described in Remark \ref{rem:embedding}.
It follows that $G$ is a strong deformation retract of $\D$.
\end{proof}

\begin{proof}[Proof\/ of\/ Theorem \ref{thm:ShellVolP2}]
To show the \emph{if}-direction, note that by Proposition \ref{pro:Consistent}
the dimer model obtained from $\H$ is consistent, and by Theorem \ref{thm:P2GivesDeformationRetract}
the dimer model is a deformation retract of $\D$. It follows that $\D$ has $2\area(\N)$-many connected
components of its complement, which is equivalent to that $\H$ has $2\area(\N)$-many cells of index zero.

To prove the \emph{only if}-direction we will rely heavily on the notation and results of \cite{FJ14}.
Let $\V$ denote the set of vertices of the hyperplane arrangement $\H$.
For $k\in \ZZ$, let $\V_k\subset \V$ denote the set of vertices of $\H$ such that
each $v \in \V_k$ bounds two cells of index $k$. Let $v\in \V_k$, and define the
\emph{oriented angle at $v$}, denoted $\theta_o(v)$, to be equal to the interior angle 
at $v$ of the polygons adjacent to $v$ with indices $k\pm 1$.
Similarly, we define the \emph{non-oriented angle at $v$}, denoted $\theta_n(v)$,
to be the interior angle at $v$ of the polygons adjacent to $v$ with index $k$.
The names stems from the fact that the boundaries of the polygons with indices $k\pm 1$
are locally oriented at $v$, while the boundaries of the polygons with index $k$ are not,
see Figure \ref{fig:indexmap}. Note that in \cite{FJ14} these angles were called \emph{inner}
respectively \emph{outer} angles; we have here chosen a name more distinct from the
terms interior and exterior angle.

Let $\Theta_o(k) = \sum_{v\in \V_k} \theta_o(v)$ and $\Theta_n(k) = \sum_{v\in \V_k} \theta_n(v)$.
It was shown in \cite[Lemma 3.2]{FJ14} that
\begin{equation}
\label{eqn:FJ14a}
2\sum_{k\in \ZZ} \Theta_o(k) =2\sum_{v\in \V} \theta_o(v) = 4\pi \area(\N).
\end{equation}
Moreover, as the sums of the exterior angles at the vertices of a cell of $\H$ of index zero is equal to
$2\pi$, we have that $\H$ has $2\area(\N)$-many cells of index zero if and only if
\begin{equation}
\label{eqn:FJ14b}
\Theta_n(-1) +  2\Theta_o(0) + \Theta_n(1)  = 4\pi \area(\N).
\end{equation}
In \cite{FJ14}, it was shown that the right hand side of \eqref{eqn:FJ14b} bounds the left hand side of
\eqref{eqn:FJ14b} for all dual hyperplane arrangements $\H$, implying that the number of cells of index zero is
at most $2\area(\N)$. By assumption, we have that \eqref{eqn:FJ14b}
holds with equality. From \eqref{eqn:FJ14a} and \eqref{eqn:FJ14b} we deduce that
\[
2\sum_{k\neq 0} \Theta_o(k) = \Theta_n(-1)  + \Theta_n(1).
\]
Thus, it suffices to show that
\begin{equation}
\label{eqn:FJ14c}
2\sum_{k > 0} \Theta_o(k) \geq \Theta_n(1)
\quad \text{and} \quad 
2\sum_{k < 0} \Theta_o(k) \geq \Theta_n(-1),
\end{equation}
with equality in both cases if and only if $|\iota(P)| \leq 2$ and each polygon $P$ with $|\iota(P)| = 2$
is a triangle. We will show the first inequality involving positive indices; the second is shown similarly.

Let us construct a number of cycles which partition on the set $\V_1$.
To begin, choose an arbitrary point $v \in \V_1$.
Since the (unique) cell $P$ of index two adjacent to $v$ is locally oriented at $v$, see Figure~\ref{fig:indexmap},
there is a unique way to depart $v$ along the boundary of $P$ in accordance with the orientations of $\H$.
Continue along the same line in $\H$ until we arrive at a second vertex $\tilde v \in \V_1$; then repeat
(c.f.\ \cite[Figure 6]{FJ14}). Notice that we locally, in the universal cover $\RR^2$, took a turn to the right at $\tilde v$.
Sinve $\V_1$ is finite, we will eventually arrive at a vertex which was already visited,
and this vertex must be $v$ since there is a unique path along which we can arrive each vertex.
Denote the obtained cycle by $C$.
If not all vertices in $\V_1$ was visited, then we choose a new starting point among the vertices not contained in 
the cycle $C$ and construct a second cycle, etc.

For an oriented, closed, piecewise linear cycle $C \subset \RR^2$
we will define the following sums. Let $R_o(C)$ respectively $L_o(C)$
denote the sum of all oriented angles at points where $C$ turn to the right respectively to
the left, and let $S_o(C)$ be the sum of all oriented angles at points where $C$ self-intersect.
If $C \subset \H$, then we define $I_o(C)$ to be the sum of all oriented angles 
at points in $\V$ where $C$ is smooth.
We define $R_n(C), L_n(C), S_n(C)$, and $I_n(C)$ similarly.
Finally, let $r(C)$ denote the number of right turns the
the cycle $C$ makes, so that $r(C) = |\V_1\cap C|$ if $C$ is constructed as above.

Since $\H$ is assumed to be generic, we have that if $v \in \V$ is such that one cycle $C$ is smooth at $v$, 
then there is exactly one other cycle $\tilde C$ passing through $v$.
Hence, to prove \eqref{eqn:FJ14c} (including the claim following the equation)
it suffices to show for each cycle $C$ constructed above that
\begin{equation}
\label{eqn:FJ14d}
2 R_o(C)  +2S_o(C)+ I_o(C) \geq R_n(C),
\end{equation}
with equality only if $C$ is the boundary of a triangle and $I_o(C) = 0$.

Since we are computing sums of angles, we can lift the cycle $C$ to one preimage in the universal cover
$\RR^2$ of $\T^2$. By abuse of notation we denote the preimage also by $C$.
Since $C$ turns only to the right, we have that
\begin{equation}
\label{eqn:FJ14e}
R_n(C)= 2 \pi d
\end{equation}
where $d$ is the \emph{turning number} of $C$. We have thus reduced to proving the
inequality \eqref{eqn:FJ14d}, with the right hand side replaced in accordance with \eqref{eqn:FJ14e},
ad that equality holds only of $C$ the boundary of a triangle and $I_o(C)=0$.

The cycle $C$ subdivides $\RR^2$ into a finite number of regions. Let us define an index map
on this subdivision, which only take the cycle $C$ into account. That is, for a region $P \in \pi_0(\RR^2 \setminus C)$
we define $\hat \iota(P)$ to be the class of $C$ in $H_1(\RR^2\setminus p, \ZZ) \simeq \ZZ$, 
where $p\in P$ is arbitrary. Notice that, if $C$ fulfills that $I_o(C) = 0$ then each $P\in \pi_0(\RR^2 \setminus C)$ is a lift of an element
of $\pi_0(\T^2\setminus \H)$ and, with slight abuse of notation, $\iota(P) = \hat \iota(P)+1$.
We consider the \emph{interior of\/ $C$} to be the union of all polygons $P \in \pi_0( \RR^2\setminus C)$
such that $\hat \iota(P) \geq 1$. The interior of $C$ need not be simply connected.
Assume that the maximum of $\hat \iota$ over $\pi_0( \RR^2\setminus C)$ is $m$. 
It suffices to show that $m=1$, that $I_o(C) = 0$ and that $C$ is a triangle.

For $j = 1, \dots, m$, let $C_j$ denote the oriented boundary of the closure of the union of all $P \in \pi_0( \RR^2\setminus C)$ with $\hat \iota(P) \geq j$. It follows that $C_1, \dots, C_m$ is a subdivision
of the cycle $C$, so that $d = d_1 + \dots + d_m$. Further more, the cycles $C_j$ intersect only at the vertices of $\H$. 
Actually, the cycle $C_j$ turns to the left at
a point $v\in \V$, if and only if $C_{j+1}$ turns to the right at $v$.
We deduce that
\begin{equation}
\label{eqn:FJ14f}
2R_o(C) +2S_o(C) +I_o(C) = \sum_{j=1}^m 2 R_o(C_j) + I_o(C_j) - 2 L_o(C_j).
\end{equation}
It was shown in \cite[Lemma 4.4]{FJ14} that, for each $j$,
\begin{equation}
\label{eqn:FJ14g}
2 R_o(C_j) + I_o(C_j) - 2 L_o(C_j) \geq 2\pi d_j.
\end{equation}
From \eqref{eqn:FJ14e}, \eqref{eqn:FJ14f}, and \eqref{eqn:FJ14g} we conclude that the inequality 
\eqref{eqn:FJ14d} holds; it remains to show that equality holds in \eqref{eqn:FJ14g} for each $j= 1, \dots, m$
only if $m=1$, and $C$ is a triangle with $I_o(C)=0$.

Consider the cycle $C_m$. Then, $C_m$ can only turn to the right, i.e. $L_o(C_m) = 0$.
Thus, $d_m$ is the number of connected components of $C_m$. Each such component
is a convex polygon, and if one such polygon has $t$ sides where $t \geq 3$
then that component contributes with $(t-2)\pi$ towards the sum $R_o(C_m)$. 
Thus, to have equality in \eqref{eqn:FJ14g} it must be that each polygon is a triangle
and, in addition, it must hold that $I_o(C_m) = 0$.

Assume now that $m \geq 2$. We will consider the cycle $D = C_{m-1} + C_m$,
which can be viewed as a cycle only turning to the right.
To obtain our contradiction, it suffices to show that 
\begin{equation}
\label{eqn:FJ14h}
2R_o(D)  + I_o(D) > 2\pi (d_{m-1} + d_m).
\end{equation}
Here, $d_{m-1} + d_m$ is the turning number of the cycle $D$. While $D$ is not necessarily
connected, it suffices for us to consider a connected component of $D$ which contain
at least one connected component of $C_m$ in its interior. We assume that this is the case;
for simplicity we will not alter the notation.
Let $t_j$ denote the number of connected components of the interior of $C_j$,
and let $g$ be the sum of the genera of the connected components of the interior of $C_{m-1}$.
We have that
\[
d_{m-1} + d_m  = t_m + t_{m-1} - g.
\]

\begin{lemma}
\label{lem:TechTheorem1}
Let $D = C_{m-1}+ C_m \subset \RR^2$ be a piecewise linear closed cycle turning only to the right
such that the following holds. Firstly, $C_m$ is a finite union of $t_m\geq 1$ triangles. Secondly, $C_{m-1}$ turns to the left at a point $v$ if and only if\/ $C_{m}$ turns to the right at $v$. 
Thirdly, each connected component of\/ $C_{m-1}$ contains at least one component of\/
$C_m$ in its interior. Then, $2 R_o(D) \geq \pi r(D)$.
\end{lemma}

\begin{proof} 
Assume that the interior of $D$ has $t_{m-1}$ connected components, and that the 
sum of the genera of the connected components of the interior of $C_{m-1}$ is $g$.
Then, as above, the turning number $d$ of $D$ is equal to $t_m + t_{m-1} - g$.
We give the proof by an induction over $t_{m}-t_{m-1}+g$. 

By assumption, $t_m \geq t_{m-1}$, thus $t_{m}-t_{m-1}+g = 0$ if and only 
if $g = 0$ and each connected component
of the interior of $C_{m-1}$ contains exactly one connected component of $C_m$.
This case is the basis of the induction.
Indeed, if $t_m= t_{m-1}=1$, then it is straightforward to verify that
$r(D)\geq 6$ and if $r(D) = 6 + j$ then
\[
2R_o(D)\geq \pi \, \left(6 + 2 j\right) \geq \pi \, r(D).
 \]
 Since both $R_o$ and $r$ are linear functions of the cycles, the induction basis follows.
 
For the induction step, assume there is a connected component
of the interior of $C_{m-1}$ which contains at least two connected components of $C_m$.
Let $P \in \pi_0(\RR^2\setminus D)$
be a polygon in the interior of $C_{m-1}$ which shares boundary with 
at least two connected components of $C_m$.
Draw a straight line segment $\ell$ through $P$ not intersecting $C_{m}$ but separating two
connected components of $C_m$.
Let us consider a new cycle $\tilde D = D + \ell - \ell$. We can view the cycle $\tilde D$ as
only turning to the right, in which case $r(\tilde D) = r(D) + 4$ and $R_o(\tilde D) = R_o(D) + 2 \pi$.
There are two cases:

\underline{Case 1:} The interior of $\tilde D$ has the same number of connected components as the interior of $D$. In this case we have the $\tilde g = g - 1$, while $t_m$ and $t_{m-1}$ are unaltered.
In particular, we find that $\tilde t_{m-1} - \tilde t_m + \tilde g = t_{m-1} - t_m + g - 1$.
Thus, the statement follows by induction in this case.

\underline{Case 2:} One connected component of the interior of $D$ was separated by $\ell$.
In this case $\tilde t_m = t_m + 1$, while $g$ and $t_{m-1}$ are unaltered.
In particular, we find that $\tilde t_{m-1} - \tilde t_m + \tilde g = t_{m-1} - t_m + g - 1$.
Thus, the statement follows by induction also in this case.
\end{proof}

\begin{lemma}
\label{lem:TechTheorem2}
Let $D\subset \RR^2$ be a piecewise linear closed cycle with turning number $d$, where $d\geq 1$, then $r(D) \geq 2d+1$.
\end{lemma}
 
\begin{proof}
We give the proof by induction. If $d=1$ then this is obvious; a closed piecewise linear cycle 
which only turns to the right twice must have negative turning number.
If $d \geq 2$, then $D$ has a self-intersection point. 
The statement now follows from that we can subdivide $D$ into two cycles $D_1$ and 
$D_2$ with turning numbers $d_1$ and $d_2 $ such that $d  = d_1 + d_2 $
and $r(D) = r(D_1) + r(D_2)-1$.
\end{proof}

We now complete the Proof\/ of\/ Theorem \ref{thm:ShellVolP2} using
Lemmas \ref{lem:TechTheorem1} and \ref{lem:TechTheorem2}.
We have that
\[
2R_o(D) \geq  \pi  r(D) \geq 2\pi (d_{m-1} + d_m) + 2\pi > 2 \pi (d_{m-1} + d_m).
\]
Thus, the strict inequality \eqref{eqn:FJ14h} holds, and we obtain our contradiction.
It follows that $m = 1$. Hence, $|\iota(P)| \leq 2$ and, as we saw earlier,
each $P$ with $|\iota(P)| = 2$ is a triangle.
\end{proof}

\begin{proof}[Proof\/ of\/ Theorem \ref{thm:DeformationRetract}]
Assume that $G$ is a deformation retract of $\overline{\C}$. 
Then, the number of faces of $G$, which is at least equal to $2 \area(\N)$,
is equal to the cardinality of $\pi_0\big(\T^2\setminus \overline{\C}\big)$.
But the cardinality of $\pi_0\big(\T^2\setminus \overline{\C}\big)$ 
is at most $2\area (\N)$ by \cite{FJ14}. It follows that  $\pi_0\big(\T^2\setminus \overline{\C}\big)$ 
has maximal cardinality.

Conversely, assume that $\pi_0\big(\T^2\setminus \overline{\C}\big)$ has cardinality
$2 \area(\N)$. Then $\pi_0\big(\T^2\setminus \D\big)$ has cardinality
$2 \area(\N)$ as well, implying that $\D$ is a strong deformation contract of $\overline{\C}$.
Also, we have that  $\H$ has exactly $2 \area(\N)$-many cells of index zero.
By Theorem \ref{thm:ShellVolP2} we find that $|\iota(P)| \leq 2$ and if $|\iota(P)| = 2$
then $P$ is a triangle. Hence, it follows from Proposition \ref{pro:Consistent}
that the dimer model $G$ obtained from $\H$ using the odd index graph and
Yang--Baxter modifications is consistent, 
and it follows from Theorem \ref{thm:P2GivesDeformationRetract}
then $G$ is a strong deformation retract of $\overline{\C}$.
\end{proof}

\begin{remark}
\label{rem:HarnackCurves}
In the examples of \cite{FU10, UY11, UY13} the dimer model $G^-$ was obtained from 
the coamoeba of the characteristic polynomial.
In general, the characteristic polynomial is not sufficiently generic.
It is known that the characteristic polynomial defines a Harnack curve \cite{KO06}.
Typically, coamoebas of Harnack curves have non-simple shells, see, e.g., \cite{Lan15}.
For an explicit example of a polygon $\N$ for which the dimer model is not a deformation retract
of the coamoeba $\overline{\C}$ of the characteristic polynomial, it suffices to consider the polygon
$\N$ with vertices $(0,0), (2,0), (0,1),$ and $(1,1)$. Then, $2 \area(\N) = 3$. 
The characteristic polynomial is
\[
f(z_1,z_2) = 1 + 2 z_1 + z_1^2 - z_2 + z_1 z_2.
\]
The shell $\H$, which is degenerate, has two cells of index zero. 
As a consequence $\RR^2\setminus \overline\C$ has at most
(and  $\RR^2\setminus \D$ has exactly) two connected components.
\end{remark}

\section{Circuits}

Theorem \ref{thm:DeformationRetract} raises the following problem: For which polygons $\N$
can we find a polynomial $f$ with Newton polygon $\N$ such that $\RR^2\setminus\overline{\C}$
has $2\area(\N)$-many connected components? This problem is open as of this writing.
It is, however, solved in the case when $f$ is supported on a (possibly degenerate) circuit.
That is, in the case when $A$ consist of four points. (The circuit is said to be degenerate if
three of the points are contained in one line.) Coamoebas of polynomials supported on circuits
was considered in the last section of \cite{FJ15}, and studied in detail in \cite{For16}.
We include here the most important implications in relation to dimer models.

\begin{proof}[Proof\/ of\/ {Corollary~\ref{cor:HVforCircuits}}]
As $f$ is assumed to be supported on a circuit
we are in the situation considered in \cite[\S 5.1]{FJ15},
where is was shown that $\overline{\C}$ has the maximal number of components
of its complement for generic coefficients of the polynomial $f$.
Hence, the result follows from Theorem \ref{thm:DeformationRetract}.
\end{proof}

\begin{remark}
\label{rem:SingularPoints}
In \cite{FHKV08} the set of critical points of the polynomial $f\in (\CC^\times)^A$ was considered.
The number of critical points in $(\CC^\times)^2$ is, by the Bernstein--Kushnirenko theorem,
equal to $2\area(\N)$. This is equal to the number of gauge groups in the
quiver theory (see \cite[\S 2.1.1]{FHKV08}). In the case when $A$ is a circuit, it was shown in \cite{For16} that (after possibly translating $\N$) the argument map restricts to a bijection between the set of critical points of $f$ and the connected components of $\T^2 \setminus \overline{\C}$. That is, these critical points are in a bijective relation with the faces of the dimer
model $G^-$, i.e., with the gauge groups in the quiver theory, see \cite{FHKV08, HV07}.
\end{remark}


\bibliographystyle{amsplain}

%

\end{document}